\documentclass[dvips,ejs]{imsart}

\RequirePackage[OT1]{fontenc}
\RequirePackage{amsthm,amsmath,natbib}
\RequirePackage[colorlinks,citecolor=blue,urlcolor=blue]{hyperref}
\RequirePackage{hypernat}

\usepackage{mathrsfs}
\usepackage{amsfonts}
\usepackage{amssymb}
\usepackage{amsmath,epsfig}

\pubyear{2009}
\volume{0}
\issue{0}
\firstpage{0}
\lastpage{0}

\startlocaldefs
\numberwithin{equation}{section}
\theoremstyle{plain}

\endlocaldefs

\newcommand{\arginf}{\mathop{\mathrm{arg\,inf}}}

\startlocaldefs
\numberwithin{equation}{section}
\theoremstyle{plain}

\newtheorem{lemma}{Lemma}[section]

\newtheorem{proposition}{Proposition}

\newtheorem{remark}{Remark}
\endlocaldefs
\newtheorem{theorem}{Theorem}
\begin{document}

\begin{frontmatter}
\title{Non-asymptotic model selection for  linear non least-squares estimation
in regression models and inverse problems
}
\runtitle{Model selection for non-least squares linear estimation}

\begin{aug}
\author{\fnms{Ikhlef} \snm{Bechar}\thanksref{t1,t2}\ead[label=e1]{ikhlef.bechar@cardiov.ox.ac.uk}
\ead[label=e2]{ikhlef.bechar@sophia.inria.fr}}




\thankstext{t1}{EPSRC Postdoctoral Research Grant, Oxford University}
\thankstext{t2}{Pulsar project, INRIA Sophia Antiplois}
\runauthor{I. Bechar}


\runauthor{I. Bechar}

\affiliation{Oxford University \& INRIA Nice Sophia Antipolis}

\address{ OCMR / FMRIB  \\
Oxford University\\
John Radcliffe Hospital\\
Oxford, OX3 9DU \\
United Kingdom \\
\printead{e1}\\
\phantom{E-mail: ikhlef.bechar@cardiov.ox.ac.uk\ } }

\address{Projet Pulsar \\
 INRIA Sophia Antipolis \\
Route des Lucioles - BP 93 \\
06902 Sophia Antipolis Cedex \\
FRANCE \\
\printead{e2}\\
\phantom{E-mail: ikhlef.bechar@sophia.inria.fr \ } }
\end{aug}

\begin{abstract}
We propose to address the common problem of linear estimation in  linear statistical
models by using a model selection approach via penalization. Depending then on the
framework in which the linear statistical model is considered namely
the regression framework or the inverse problem framework, a data-driven model
selection criterion is obtained either under general assumptions, or under the mild
assumption of model identifiability respectively. The proposed approach was stimulated
by the important recent non-asymptotic model selection results due to Birg\'e and Massart mainly
\cite{birge_massart_2007}, and our results in this paper, like theirs,
are non-asymptotic and turn to be sharp.

Our main contribution in this paper resides in the fact that these linear estimators
are not necessarily least-squares estimators but can be any linear estimators.
The proposed approach finds therefore potential applications in countless
fields of engineering and applied science (image science, signal processing,
applied statistics, coding, to name a few) in which one is interested in
recovering some unknown vector quantity of interest as the one, for example,
which achieves the best trade-off between a term of fidelity to data,
and a term of regularity or/and parsimony of the solution. The proposed
approach provides then such applications with an interesting model 
selection framework that allows them to achieve such a goal.

\end{abstract}


\end{frontmatter}

\section*{Summary of the main results}

Let us fix first some of the notations that we shall use 
in the sequel. So, we consider the problem estimation of a
vector quantity $\beta = (\beta)_{k=1,p}$ lying in some
$p-$dimensional Euclidean subspace of $\mathbf{R}^p$
(though our results extend easily to infinite dimensional Hilbert spaces)
endowed with the traditional Euclidean norm $\|\cdot\|$ defined as follows
\begin{equation}
\|\mu\|^2 = \sum_{k=1}^p \mu_k^2, \, \forall \mu =(\mu_k)_{k=1,p} \in \mathbf{R}^p \nonumber
\end{equation}
We shall also use the same notation, i.e., $\|A\|$ to mean the Euclidean norm of
any $q$ by $q$ real matrix $A$ for some $q \in \mathbb{N}$, i.e.,
\begin{equation}
\|A\|^2 = \sum_{i=1}^q \sum_{j=1}^q a_{i,j}^2,  \,  \forall \, A=(a_{ij})_{i=1,q; j=1,q}, \,\, a_{ij} \in \mathbf{R},  \, \, \forall  \, i,j=1,\cdots,q \nonumber
\end{equation}
Sometimes, we use for some $q \in \mathbb{N}$ the notation  $I_q$ to mean the $q$ by $q$
identity matrix, and the notation $D\{\eta_k\}_{k=1,q}$ to mean
a $q$ by $q$ diagonal matrix with diagonal elements $\eta_k, k=1,\cdots,q$.
When we write for some matrix $A$ the following $A^{\dag}$, we mean
the Moore-Penrose pseudo-inverse of matrix $A$.

With this being said, we consider the classical linear gaussian regression model $y= X\beta + R z$,
and we assume that one is given a collection of linear estimators of the
solution $\beta$ as follows $\mathcal{L} = \{\hat{\beta}_{m}:= \Psi_m y, m\in\mathcal{M}\}$
parameterized by some set of models (parameters) $\mathcal{M}$, that we
consider finite in this paper though our results generalize easily
to the case when $\mathcal{M}$ is countable-infinite. Such a collection of
linear estimators can be obtained for instance by considering one (or more)
of the frameworks presented in section \ref{sec1}. The goal is then
to select among such a finite collection of linear estimators one estimator
of $\beta$ with the lowest quadratic risk (our results extend easily to other
performance measures such as the weighted quadratic risk or the Mahalanobis risk).
To do this, we firstly assume that when matrix $X$ is rank-deficient, the solution of interest
$\beta$ is linearly-identifiable in the linear model above; which means that one knows
a-priori some $p$ by $n$ real matrix $\mathcal{K}$ such that one can write
$\beta - \mathcal{K} X \beta =0$ \footnote{When $X$ is of full rank, then
one takes $\mathcal{K} := (X^T X)^{-1} X^T$.}, and we adopt a model 
selection procedure from a non-asymptotic point of view via penalization. Thus,
we propose to select the estimator of $\beta$ by minimizing over $\mathcal{M}$
a penalized criterion of the form
\begin{equation}
\nonumber
\textrm{Crit}(m) = y^T \Big ( \Psi_m^T \Psi_m  - 2 \mathcal{K}^T  \Psi_m) y + \textrm{pen}(m)
\end{equation}
where $\textrm{pen}(m), m \in \mathcal{M}$ stands for some penalty function that depends
exclusively upon a given model $m$ through matrix $\Psi_m$ but not upon data that
we propose to address in the remainder of this paper, and we refer to the estimator
denoted by $\tilde{\beta} = \hat{\beta}_{\tilde{m}}$ with $\tilde{m} = \arginf_{m\in\mathcal{M}} \textrm{Crit}(m)$
as the penalized estimator of $\beta$. Then, we show that when the size of the
collection of estimators $\mathcal{L}$ is not too
great, one can set the penalty $\textrm{pen}(m)$ in such a way to enforce an oracle inequality of the form
\begin{equation}
\mathbb{E}\Big[\|\tilde{\beta} - {\beta} \|^2\Big] \leq K \min_{m\in\mathcal{M}}\mathbb{E}\Big[\|\hat{\beta}_m - {\beta} \|^2\Big] + C  \nonumber
\end{equation}
for some universal additive constant $C$ and some multiplicative constant $K$
which depends on the complexity of the model collection $\mathcal{M}$. \\
More generally, especially when the size of the collection $\mathcal{L}$ can be great,
then we show that another choice of  $\textrm{pen}(m)$ that takes into account the
complexity of each model with respect to the collection of models $\mathcal{M}$--in a sense
that we shall specify in the sequel-- warrants to obtain a sharp inequality of the form
\begin{equation}
\mathbb{E}\Big[\|\tilde{\beta} - {\beta} \|^2\Big] \leq K \min_{m\in\mathcal{M}}\Big\{\mathbb{E}\Big[\|\hat{\beta}_m - {\beta} \|^2\Big] +  T_m \Big\} + C  \nonumber
\end{equation}
with $K$ and $C$ standing for some reasonable universal constants, and  $T_m$
stands for some reasonable per-model positive quantity which is related to
the complexity of a given model $m$ with respect to the model
collection $\mathcal{M}$.

\section{Introduction}
\label{intro}
We assume the correlated linear statistical model
\begin{equation}
\label{model}
  y =  X \beta +  R  z
\end{equation}
where $y$ is a $n$-dimensional vector of observations,
$\beta$ is the $p-$dimensional vector parameter of
interest, $X$ and $ R$ are $n$ by $p$ and $n$ by $n$
design (known) matrices respectively, and we take $z$ to
be a $p-$dimensional vector of independent and identically distributed (i.i.d.)
random variables $N(0,1)$. We are then interested in estimating the vector $\beta$
(resp. the response $X\beta$) given a single realization of the vector $y$
by adopting a model selection approach via penalization.

We propose to address such an estimation problem under rather general assumptions about
the model parameters $X$, $R$, $n$ and $p$ and the solution $\beta$. So, we regard $\beta$
as an unknown deterministic vector quantity, and we assume that matrix $R$ is a general
$n$ by $n$ matrix, matrix  $X$ can be a well or an ill-conditioned  matrix, the number of the observations (n) can be greater,
equal or smaller than the size of the vector $\beta$ (p), which therefore includes the case known as the
"$n<p$" set up where on seeks to recover high resolution or sparse vector quantities
$\beta$ by using only a few  measurements \cite{donoho_2006, candes_plan_2009}.

Our results in this paper are non-asymptotic; this means that we propose to work with the
values of the parameters $R$, $n$ and $p$ of the linear model as they are, and we allow
the number of models of the solution $\beta$ to depend freely upon the of the solution ($p$). This viewpoint as
initiated in model selection by Barron, Birg\'e and Massart \cite{barron_birge_massart_1999} then refined by Birg\'e and Massart
\cite{birge_massart_2001, birge_massart_2007}, needs to be contrasted with the
asymptotic point of view \cite{akaike_1973, mallows_1973, schwartz_1978, shibata_1981, johnstone_1999} which
considers for example that the number of the observations goes to infinity or that the noise magnitude
goes to zero while the number of models remains fix. As we shall see in the remainder, useful model selection criteria are
directly connected to the complexity of a family of models, and by the latter, we roughly mean how
large a model collection is; compared with  the size ($p$) of the vector $\beta$ \footnote{It is interesting to note that we do not
say with ($n$) because, anyway, $p$ and $n$ are related since they have
to satisfy generally that $n \succeq O(\frac{S}{\log[p/n]})$ where $S$
stands for the maximum number of the non-zero components of the vector
$\beta$ or of its coefficients with respect to some fixed orthonormal basis,
otherwise it is not generally possible to recover efficiently $\beta$ even when
the latter is some high-resolution signal\cite{candes_tao_2007b}.}.
Such a non-asymptotic property turns indeed to be very precious in many practical applications,
mainly those which seek to recover a given vector quantity of interest by using a large library
of models. We refer the interested reader to \cite{birge_massart_2007, massart_2007} for
a more thorough discussion on the topic.

Non-asymptotic model selection by using penalization has become
an important trend in statistical estimation in regression, and 
will certainly continue in fascinating  many researchers either 
in statistics or in the engineering field and applied science for a long time. 
Early works in the field appeared indeed in the early nineties due to Barron and
colleague \cite{barron_1991} for discrete models, and extensions to continuous models 
were proposed by Barron, Birg\'e and Massart \cite{barron_birge_massart_1999}. Aware 
of the pioneering works of Talagrand on concentration inequalities \cite{talagrand_1996} \footnote{We refer the dear reader to \cite{massart_2000, massart_2007} for two beautiful lectures on the topic of concentration inequalities and 
their application in model selection.}, Birg\'e and Massart \cite{birge_massart_2001,  birge_massart_2007} improved on Talgrand's works
and refined the model selection approach in \cite{barron_birge_massart_1999} 
and addressed nicely the problem of the estimation of the mean of a gaussian process 
in homoscedastic regression models when the variance is known (or estimated off-line), 
while taking into consideration the richness (complexity) of a collection of models,
and which leads in some cases (when the model collection is not too big)
to sharp oracle inequalities. Baraud \cite{baraud_2000, baraud_2002} and Baraud
and colleagues  \cite{baraud_compte_viennet_2001}  proposed many extensions to
the aforementioned works that generalize the approach to non-gaussian homoscedastic
statistical models; and under some mild assumptions about noise moments, amazingly
they obtained near-oracle results. Baraud, Giraud and Huet \cite{baraud_giraud_huet_2009} proposed recently
penalized model selection criteria that are able to estimate the mean of gaussian homoscedastic models
even when the variance is unknown, and they proved results for both the quadratic
risk and the Kullbak risk. Very recently, Gendre \cite{gendre_2008, gendre_2009} extended
their results in the case of a simultaneous estimation of the mean and the variance
in gaussian heteroscedastic models. The latter work is probably the closest to our
present work since it can handle heteroscedastic data as well, however, 
a subtle nuance exists between the two methods. Indeed while our method works with 
any linear estimators, the method of the author is based solely on least-squares estimators; which means that for
each model $m$ of the model collection (assumed to be some family of subspaces to which one can associate orthonormal bases),
an estimator is constructed as the closest $p-$dimensional real vector to data (in the sense of the
squared error) by assuming that current model $m$ is true. Though we have to admit
that such a work constitutes a pretty important contribution to the field of model
selection in correlated regression models, the approach of the author might suffer
for yielding the expected results for some instances of the vector $\beta$ and of the
noise matrix $R$ especially when $\beta$ has many nonzero coefficients
with respect to any basis among the family of bases of the solution,
and/or matrix $R$ plays a preponderant role in the formula of the quadratic risk.
However, when such a \emph{least-squares} restriction is relaxed to allow
the construction of non least-squares linear estimators of the solution
(which is the case here), interestingly, strong priors about the sought solution
$\beta$ can be inserted in a plenty of ways into its linear estimators
(for instance, as the best trade-off between fidelity to data and regularity),
consequently, one might limit considerably the influence  of matrix $R$ on the recovery process\footnote{We would like then
to point out that this paper is devoted to linear estimation by model selection
in a broad context, and another paper \cite{bechar_2009a} which deals specifically
with the problem of construction of non least-squares estimators that
can take into account to some extent data heteroscedasticity in the goal
of achieving more interesting balances between the bias and the variance terms
in the formula of the quadratic risk is currently in the writing
process by the author in the spirit of the present work.}.
Interestingly, this finds countless applications  in the fields of
engineering and applied science--witness the broad success
of Bayesian regularization frameworks in signal and image restoration.
We would not close this rather modest overview of related works without
probably pointing out one thriving field of non-asymptotic estimation
in regression pioneered mainly by Cand\`es and collaborators \cite{candes_tao_2007, candes_tao_2007b, candes_watkin_boyd_2008, candes_plan_2009} and currently explored by many research groups in statistics and
engineering disciplines (image and signal processing, information and coding theory, etc.)
and which is known as sparse statistical estimation or compressive sensing. In the latter,
one tries to recover some vector quantity of interest assumed to be sparse or
admits a parsimonious representation in a fixed orthonormal basis
by considering criteria based on the $\ell_1$-norm minimization by using
linear programming concepts. The authors obtained consequently under some
minimal assumptions about the model near-oracle inequalities of
their estimation method.


Before describing in details the main model selection results in this
paper, we would like to open here a discussion to briefly point out
our point of view regarding which performance measure of an estimator of
$\beta$ is better suited for some application; for sake of making our
approach in this paper as much clear as possible, and for helping the
dear reader set up more easily his/her model selection framework. In fact,
one has to distinguish generally between the following two frameworks in
 which the linear statistical model (\ref{model}) can be considered which we briefly review:
\begin{enumerate}
\item {\bf The linear regression framework}: where the estimation of $X\beta$ is generally, though not
always\footnote{because one would also use the regression framework to select an estimator of $\beta$, but it is to use
with some care to avoid any bad situation of selecting an estimator $\hat{\beta}$ such that
$\mathbb{E}\big[\|X\hat{\beta} - X\beta\|^2\big] \approx 0$, but $\mathbb{E}\big[\|\hat{\beta}- \beta \|^2 \big]\gg 0$.},
the main concern for the statistician. Hence, the performance measure of any estimator $\hat{\beta}$ of $\beta$ which is used
in this case is the predictive risk given by $\mathbb{E} \big [ \| X\hat{\beta} - {X\beta} \|^2 \big]$.
\item {\bf The linear inverse problem framework}: where the estimation of $\beta$ is the main concern for the statistician.
Hence, the performance measure of any estimator  $\hat{\beta}$ of $\beta$ which is used in this case is the quadratic risk
given by $\mathbb{E} \big [ \| \hat{\beta} - {\beta} \|^2 \big]$.
\end{enumerate}
However, one shows easily that w.l.g., even to consider a collection of linear estimators
of $X\beta$ as follows
\begin{equation}
\nonumber
\{\widehat{X\beta}_{m}:=  X\hat{\beta}_{m} := X \Psi_m y, m\in\mathcal{M}\}
\end{equation}
and use as a figure of merit of any estimator $\widehat{X\beta} := X \hat{\beta}$ of
$X\beta$ its quadratic risk, which is simply the predictive risk of
the estimator $\hat{\beta}$ of ${\beta}$, one then finds oneself in the
presence of the model selection problem we stated earlier in this section. Consequently,
we shall address in the remainder the two problems, namely the linear regression problem
and the linear inverse problem, in a same framework. However, in contrast to the
regression set up, model identifiability might be an issue to consider in
the inverse problem set up to derive useful model selection procedures, so we shall also
provide some useful ideas that help overcome such an issue. \\
Regarding now the appropriateness of either performance measure of any estimator of $\beta$, it turns out indeed
that in many engineering domains, the vector quantity of interest that has direct application is $\beta$ and not $X\beta$
\cite{candes_tao_2007b}. In this case, using the quadratic risk as a figure of merit of an estimator of $\beta$
makes more sense than using the predictive risk. In image reconstruction/restoration
applications for instance, $\beta$ represents an image (e.g. an MRI scan of the heart or the brain of a
patient), $X$ then models the design matrix of the imaging device (e.g. the MRI scanner), and $Rz$
models the stochastic errors of the imaging device. Obviously, the vector quantity of interest in
this case is the image $\beta$ that one would need to reconstruct for further use (for cardiac or brain
diagnosis for example). However, when one is more interested in the estimation of the system response $X\beta$
than $\beta$ itself, then using the predictive risk as a figure of merit of an estimator sounds
more interesting in this case. As some illustrative examples, one can mention the variable selection problem
that we shall discuss in section \ref{sec1}, and the reconstruction problem (for filtering or recognition purposes,
etc.) by using a finite library of vector primitives. In the latter, one has generally an a-priori linear model of
some vector quantity of interest $u$ (a signal, an image, a curve, a shape, etc.) as follows $u = X\beta$,
where $\big \{X_k, k=1,\cdots,p\big\}$ stand for some set of vector primitives (predictors)
which, depending on the application, can be for example eigen objects (e.g. eigen images, eigen shapes, wavelets)
or simply some database of generic objects,  and $\beta$ stands for the vector of the coefficients
of the linear combination of such vector primitives. One's goal is then
to estimate the vector of the coefficients $\beta$ in such a way to achieve the
lowest predictive error. To cut a long story short, depending on the application, one may
find good reasons to prefer the use of one performance measure from
the use of its alternative (see for example subsection \ref{sub_eq} for another
reason that might justify the use of the predictive risk).

Having said this, the rest of the paper is organized as follows. In section \ref{sec1},
we set up our model selection framework and we describe some applications that
can be expressed in terms of our framework.  In section \ref{sec2}, we derive the
main data-driven model selection results in this paper, and we provide
some useful clues that help choose the penalty function and its parameters.
In section \ref{sec3}, we discuss some practical solutions
that can lead to overcome the identifiability issue of model (\ref{model}) when the
rank of matrix $X$ is smaller than $p$. Section \ref{sec5} is
devoted to the numerical experiments that show the performances of
the approach for some application examples. Finally,
a general discussion about the proposed approach and its future
extensions concludes this papers.

\section{Collection of linear estimators of $\beta$}
\label{sec1}
We assume that one is given typically a finite collection of linear
estimators of the solution $\beta$ which is parameterized by some finite
set of models (or parameters) $\mathcal{M}$ as follows
\begin{equation}
\label{lin_estimator}
\mathcal{L} = \{\hat{\beta}_{m}:= \Psi_m y, m\in\mathcal{M}\}
\end{equation}
with $\Psi_m$ standing for some $p$ by $n$ matrix for all $m\in\mathcal{M}$,
and the goal is to select among such a family of linear estimators $\big\{\hat{\beta}_{m}, m\in\mathcal{M}\big\}$
one estimator of $\beta$  with the lowest quadratic risk.

Concerning the way in which these linear estimators are constructed, this
depends generally on the application and on the prior that one has about
the solution $\beta$, and some examples of their construction encountered
in countless engineering domains (signal/image processing, computer
vision, applied statistics, to name a few) are highlighted below.

\subsection{Reconstruction/Restoration by regularization}
In many image and signal reconstruction/restoration applications,
for recovering some vector quantity of interest namely $\beta$
from an observation of model (\ref{model}), one proceeds
by solving an unconstrained quadratic problem of the form
\begin{equation}
\label{tikonov}
\Big(y - X\hat{\beta}\Big)^T P \Big(y- X\hat{\beta}\Big) + \hat{\beta}^T H \hat{\beta} \rightarrow \min_{\hat{\beta} \in \mathbf{R}^p} 
\end{equation}
where $P$ and $H$ stand for two given $n$ by $n$ and $p$ by $p$ symmetric
matrices respectively, both considered generally to be positive semi-definite.
The two competing terms $\Big(y - X\hat{\beta}\Big)^T P \Big(y- X\hat{\beta}\Big)$ and $\hat{\beta}^T  H \hat{\beta}$
in (\ref{tikonov}) in (\ref{tikonov}) are generally referred as the data fidelity term,
and the regularity term respectively. As an illustrative example, when $\beta$ is assumed to be
some smooth vector quantity, then $H$ is taken generally to be some
regularization operator (e.g. a high-pass filter such as a differential operator or any linear
combination of differential operators of different orders, a band-pass filter, etc.) in order to
enforce smoothness of the recovered solution, and $P$ is generally taken as the inverse of the
covariance matrix of the observations, i.e., $P = \big(R^TR\big)^{-1}$ or simply as the $n$ by $n$
identity matrix, i.e., $P = I_n$ . This is also known in applied science as the Tikhonov regularization
problem \cite{tikhonov_1977} and which has Bayesian interpretation (maximum of posterior $\log$-likelihood of data in a gaussian set up).
Solving now for $\hat{\beta}$ in (\ref{tikonov}) gives
\begin{equation}
 \hat{\beta} = \big( X^T P X + H \big )^{\dag} X^T P y \nonumber
\end{equation}
However, it is quite common that matrix $H$ and  probably matrix $P$ too
(useful if one wants for example to test the performance of the linear estimators
against various Mahalanobis distances between data and an estimator's response)
depend on some parameters (referred in engineering as regularization parameters) that are difficult to tune optimally
for a particular application. Therefore, one assumes a finite collection of parametric
couples of instances of the matrices $P$ and $H$ as follows $\{ \big( P_m, H_m\big), m \in \mathcal{M}\}$,
and by putting  for all $m \in \mathcal{M}$
\begin{equation}
\Psi_m := \Big( X^T P_m X + H_m \Big )^{\dag} X^T P_m \nonumber
\end{equation}
one obtains finally a finite parametric collection of linear estimators of $\beta$
of the form $\Big \{ \hat{\beta}_m := \Psi_m y, m \in \mathcal{M}\Big \}$. The goal is
then to select among such a parametric collection of linear estimators of $\beta$
one estimator with the lowest risk. One then finds oneself in the presence
of the model selection framework that we set up in the beginning of this section.

\subsection{Optimal representation in a family of orthonormal bases}
Some applications that attempt to estimate $\beta$ from an observation of
model (\ref{model}) assume that $\beta$ has a sparse representation in
a given family of orthonormal bases  in $\mathbf{R}^p$  of different complexities
(i.e., dimensions) $\Lambda = \{\Lambda_m, m \in \mathcal{M}\}$ (e.g. trigonometric
bases, a wavelet family, etc.), and let us denote by $\overline{\Phi}_m$ the matrix
which rows correspond to the respective vectors of the complement basis in $\mathbf{R}^p$
of the orthonormal basis ${\Lambda_m}$, for all  $m \in \mathcal{M}$. So, these
applications proceed indeed by solving for all $m \in \mathcal{M}$ a constrained
quadratic program of the form
\begin{equation}
\label{orth}
\Big(y - X\hat{\beta}_m\Big)^T P_m \Big(y- X\hat{\beta}_m\Big) \rightarrow \min_{\hat{\beta}_m/\overline{\Phi}_m \hat{\beta}_m = 0}
\end{equation}
where $P_m$ stands for a $n-$dimensional symmetric matrix. Let us put
\begin{equation}
C_m = (X^T P_m X + \overline{\Phi}_m^T \overline{\Phi}_m )^{\dag} \nonumber
\end{equation}
One checks that the optimal estimator $\hat{\beta}_m^*$ with respect to $m \in \mathcal{M}$
is given by
\begin{equation}
\hat{\beta}_m =  C_m \bigg[ I_p - \overline{\Phi}_m^T  \Big( \overline{\Phi}_m C_m \overline{\Phi}_m^T\Big)^{\dag} \overline{\Phi}_m C_m \bigg] X^T P_m y \nonumber
\end{equation}
moreover, if matrix $\big(X^T P_m X + \overline{\Phi}_m^T \overline{\Phi}_m\big)$ is of full rank (i.e., $p$),
then such an estimator is the unique solution of (\ref{orth}). Let us now put for all $m \in \mathcal{M}$
\begin{equation}
\nonumber
\Psi_m  :=  C_m \bigg[ I_p - \overline{\Phi}_m^T  \Big( \overline{\Phi}_m C_m \overline{\Phi}_m^T\Big)^{\dag} \overline{\Phi}_m C_m \bigg] X^T P_m
\end{equation}
hence one obtains a finite collection of linear estimators of $\beta$
which is given  by $\Big\{\hat{\beta}_m := \Psi_m y, m \in \mathcal{M}\Big\}$,
among which one wants to select one estimator with the lowest risk. One
retrieves again our model selection framework.

\begin{remark}
The formula of the solution $\hat{\beta}_m$ of (\ref{orth}) and all those that shall come up later
in this section remain valid for arbitrary matrices $P_m$ and $\overline{\Phi}_m$,
and it is not definitely necessary to assume that the row vectors of $\overline{\Phi}_m$
are orthonormal (hence when $\Phi_m$ is orthogonal, one may replace $\overline{\Phi}_m$ with
$I_p - \Phi_m^T \Phi_m$). Please note that one can approximate the formula
of $\hat{\beta}_m$ by a simpler formula for some large enough positive
number $\mu$ as follows $\hat{\beta}_m \approx (X^T P_m X + \mu \overline{\Phi}_m^T \overline{\Phi}_m )^{\dag} X^T P_m y$.
\end{remark}

\subsection{Optimal representation in a family of orthonormal bases with regularization}
It happens that the solution $\beta$ that one looks to estimate from
an observation of model (\ref{model}) has a sparse representation
in some family  $\Lambda $ of orthonormal bases of some subspaces in
$\mathbf{R}^p$ which we write as $\Lambda = \big\{ \Lambda_{m'} , m' \in \mathcal{M}^{orth} \big\}$
and has some regular (e.g. smooth) structure.  So, let us denote for all
$m' \in \mathcal{M}^{orth}$ by ${\Phi}_{m'}$ the matrix which rows correspond
to the respective vectors of the orthonormal basis ${\Lambda_{m'}}$, and
by $\overline{\Phi}_{m'}$  the matrix which rows correspond
to the respective vectors of the complement basis $\overline{\Lambda}_{m'}$
in $\mathbf{R}^p$ of ${\Lambda_{m'}}$ and put  $\overline{\Lambda} = \big\{ \overline{\Lambda}_{m'} , m' \in \mathcal{M}^{orth} \big\}$.
The goal is then to recover the solution of $\beta$ in the orthonormal family
$\Lambda$ as parsimoniously as possible, while  imposing either on $\beta$ or on the vector
of coefficients of $\beta$ in any orthonormal basis $ \Lambda_{m'}  \in \Lambda $
to have a regular profile (e.g. smoothness), but the latter
is generally difficult to guess beforehand. Hence, one would like
to consider on top of $\Lambda$ a finite parametric collection $\mathcal{H}$
of regularizing operators with increasing regularization powers as follows $\mathcal{H} = \big \{H_{\theta}, \theta \in \Theta \big\}$,
and depending on whether the smoothing is applied on $\beta$
directly or on its vector of coefficients with respect to
any  orthonormal basis $ \Lambda_{m'}  \in\Lambda$, an operator
$H_{\theta} \in \mathcal{H}$ may depend or not on a basis $ \Lambda_{m'}$.
Nonetheless, for the sake of simplicity and without
loss of generality (w.l.g.) (see remark \ref{remark0} below), one can consider
that for all $\theta \in \Theta$, $H_{\theta}$ is $p$ by $p$ symmetric matrix and that
one is given a family of models (quadruplets) of the form $\big\{(\Lambda_{m}, \overline{\Lambda}_{m},H_{m}, {P}_m),
m \in \mathcal{M}\big\}$ where for all $m\in \mathcal{M}$, $\Lambda_{m} \in \Lambda$, $\overline{\Lambda}_{m} \in \overline{\Lambda}$,
$H_{m} \in \mathcal{H}$, and ${P}_m$ is some $n$ by $n$ symmetric matrix which
enforces closeness of an estimator to data. Then with respect to all $m \in \mathcal{M}$,
one solves a constrained quadratic program of the form
\begin{equation}
\label{orth_reg}
\Big(y - X\hat{\beta}\Big)^T P_m \Big(y- X\hat{\beta}\Big) + \beta^T H_m \beta \rightarrow \min_{\hat{\beta}/\overline{\Phi}_m \hat{\beta} = 0} \nonumber
\end{equation}
Let us put for all $m \in \mathcal{M}$
\begin{equation}
C_m = (X^T P_m X + \overline{\Phi}_m^T \overline{\Phi}_m + H_m )^{\dag} \nonumber
\end{equation}
One checks (see the proof in the appendix section) that the optimal estimator $\hat{\beta}_m^*$ with respect to
$m \in \mathcal{M}$ is given by
\begin{equation}
\hat{\beta}_m =  C_m \bigg[ I_p - \overline{\Phi}_m^T  \Big( \overline{\Phi}_m C_m \overline{\Phi}_m^T\Big)^{\dag} \overline{\Phi}_m C_m \bigg] X^T P_m y \nonumber
\end{equation}
and one can approximate the latter for a large enough positive number $\mu$ as follows
\begin{equation}
\hat{\beta}_m \approx (X^T P_m X + H_m + \mu \overline{\Phi}_m^T \overline{\Phi}_m )^{\dag} X^T P_m  y \nonumber
\end{equation}
Let us now put for all $m \in \mathcal{M}$
\begin{equation}
\nonumber
\Psi_m := C_m \bigg[ I_p - \overline{\Phi}_m^T  \Big( \overline{\Phi}_m C_m \overline{\Phi}_m^T\Big)^{\dag} \overline{\Phi}_m C_m \bigg] X^T P_m \nonumber
\end{equation}
One obtains in the end the following finite collection of linear estimators of $\beta$: $\Big\{\hat{\beta}_m := \Psi_m y, m \in \mathcal{M}\Big\}$ among
which one would want to select one estimator with the lowest risk. One retrieves again our model
selection framework.
\begin{remark}
\label{remark0}
When the smoothness constraint is imposed rather on the coefficients
of an estimator $\hat{\beta}_m$ in a given orthonormal basis
$\Lambda_{m}  \in \Lambda $, one can proceed in the same way as
above by defining for all $m \in \mathcal{M}$ the new filter $H_m$
as follows  $ H_m := \Phi_m^T F_m \Phi_m$ with $F_m$ standing for
the filter which operates on the vector of the coefficients
of $\hat{\beta}_m$ with respect to $\Lambda_{m}$, and $ H_m$
standing for the new filter that operates directly on the
reconstructed solution $\hat{\beta}_m$.
\end{remark}

\subsection{Optimal linear filtering}
Linear filtering (low-pass, high-pass, band-pass, band-stop, etc.)
is an important pre-processing task in various signal and image processing
applications. For instance, low pass filtering (e.g. gaussian filtering)
has become a standard step of the image processing chain which aims
at removing white noise from a noisy image $y$ in order to improve
its quality (measured generally  as peak signal-to-noise ratio (PSNR))
and simplify its further analysis and interpretation. However, linear filters
depend generally on some parameters (e.g. a filter's bandwidth) which are
difficult to tune optimally for a particular signal or an image. Hence one would
need to consider a finite collection of parametric linear filters as follows
$\{\Psi_{m}, m \in \mathcal{M}\}$  from which one would like to select the one
with the best parameter to apply on the noisy image/signal $y$. Such a problem can be
formulated indeed in terms of our model selection framework by
considering a finite collection of linear estimators of
the noiseless signal or image $\beta$ as follows
$\{\hat{\beta}_m:=\Psi_{m}y, m \in \mathcal{M}\}$,
therefore the goal amounts to selecting one estimator
of $\beta$ with the lowest quadratic risk.

\subsection{Variable selection in regression}
In this problem, one commonly assumes a collection of possible configurations $ \mathcal{C}= \Big \{ \nu_m, m \in \mathcal{M}\Big \}$ of
$\beta$, where each configuration $\nu$ stands for a binary vector of the same size as $\beta$, and a component $\nu_k = 1$
means  that the predictor $X_k$ (i.e., the $k-th$ column of matrix $X$) is not part of the regression model
(i.e., $\nu_k=1$, $\beta_k = 0$), otherwise it is said to be part of the regression model (i.e., $\nu_k=0$, $\beta_k \neq 0$),
and the goal is to figure out the configuration of $\beta$ that achieves lowest predictive error (see
below for more insight on the use of the predictive risk), in other words, the most influential variables
(i.e., those with most explanatory power) in model (\ref{model}) among the set of variables $\{X_k, k =1,\cdots,p\}$.
So, let us denote for all $ m \in \mathcal{M}$ by $N_m$ the $p-$dimensional diagonal matrix such that $N_m(k,k) = \nu_k, k=1,\cdots,p$.
Hence one proceeds by minimizing with respect to each configuration $\nu_m \in \mathcal{C}$  a
constrained quadratic program of the form
\begin{equation}
\label{orth2} \nonumber
\Big(y - X\hat{\beta}_m\Big)^T P_m \Big(y- X\hat{\beta}_m\Big) \rightarrow \min_{\hat{\beta}/ N_m \hat{\beta} = 0} \nonumber
\end{equation}
for some $n-$dimensional symmetric matrix $P_m$ to achieve finally a collection of linear estimators
of $\beta$ as follows: $\Big\{ \hat{\beta}_m := \Psi_m y, m \in \mathcal{M}\Big\}$ where for all
$m \in \mathcal{M}$, one has $\Psi_m$ which is given by
\begin{equation}
\hat{\beta}_m =  C_m \bigg[ I_p - N_m  \Big( N_m C_m N_m\Big)^{\dag} N_m C_m \bigg] X^T P_m y \nonumber
\end{equation}
with
\begin{equation}
C_m = (X^T P_m X + N_m)^{\dag} \nonumber
\end{equation}\\
We would like to emphasize that our goal in the remainder is not to address
the problem of construction of the matrices $P_m$, $\Phi_m$, $H_m$, $N_m$ for
a given application since they are application-dependent and they are
constructed from the a-priori knowledge about the solution as we already
mentioned it above \footnote{in fact, the experts in these domains have already
done a great job in this respect, and a overview of the subject
in this paper would only be a weak copy of their previous findings \dots}.
Nevertheless, our goal--once a finite  collection of them has been chosen--is to provide the
user with an efficient model selection tool that allows him/her to choose the
almost best ones to recover the solution.

Before starting in addressing such a model selection issue, we would like
to enunciate upfront the following proposition which gives the exact
formula of the ideal linear estimator of $\beta$ with respect
to the quadratic risk (see formula (\ref{risk}) in section \ref{sec2}).
\begin{proposition}
\label{prop1}
The optimal linear estimator $\hat{\beta}^*$ that minimizes in $\mathbf{R}^p$
the formula of the quadratic risk is given by
\begin{equation}
\nonumber
\hat{\beta}^* = \beta (X\beta)^T\Big( (X \beta)(X\beta)^T+  R R^T \Big)^{\dag} y
\end{equation}
and its quadratic risk is given by
\begin{multline}
\nonumber
\mathbb{E}\Big[\| \hat{\beta}^* - {\beta} \|^2\Big] = \| \beta\|^2 \Big [ \\
(X\beta)^T {\Big( (X \beta)(X\beta)^T + R R^T \Big)^{\dag}}^T (X\beta)
(X\beta)^T {\Big( (X \beta)(X\beta)^T + R R^T \Big)^{\dag}} (X\beta) \\
-2 (X\beta)^T {\Big((X \beta)(X\beta)^T + R R^T \Big)^{\dag}}(X\beta)
+ \Big \| {\Big( (X \beta)(X\beta)^T + R R^T \Big)^{\dag}}(X\beta) R  \Big\|^2 +1 \Big ]
\end{multline}
\end{proposition}
\begin{proof}
The proof of this proposition consists of a simple minimization
of formula (\ref{oracle}) of section \ref{sec3} which thus amounts
to solving an unconstrained quadratic program.
\end{proof}
\begin{remark}
As an aside, it is interesting to note that proposition \ref{prop1} says among other
things that the solution with the lowest risk of model (\ref{model})
corresponds to the least-norm solution of $y' = X\beta$, which is nothing else
than $X^{\dag} y' \equiv X^{\dag} X \beta$. However, the latter does not
correspond generally to the solution that one is looking to estimate, hence
the estimation of the actual solution would have unavoidably a bigger
risk.
\end{remark}
Though such a result of proposition (\ref{prop1}) is of little use in practice
since it says that such an ideal filter depends on the unknown $\beta$ itself, nevertheless,
it might provide  some useful hints for the construction of potential linear models of
the solution for some applications. Moreover, when the object $\beta$ that one would
like to estimate represents some object among a finite library of objects,
then spanning such a library once allows to construct potential linear estimators of the
solution among which the ideal one exists.  For instance, in recognition or tracking tasks
in computer vision, the unknown $\beta$ might model a given physical object of interest
that was previously extracted from some image for example by using any segmentation method, and
one wants actually to recognize it among a finite library of objects $\mathcal{O}= \{ O_m, m \in \mathcal{M}\}$.
Then to each object $O_m, m \in \mathcal{M}$, one associates a $p-$dimensional vector of features
$\mathcal{\beta}_m$ (for instance, some discretization of the contour of $O_m$),
matrix $X$ stands for some linear geometric transformation (rigid, affine, projective, etc.)
applied onto the object, $X\beta$ represents the object $\beta$ after deformation and $Rz$
represents noise (and probably the errors due to the linear approximation of the deformation).
Hence, one might consider the finite collection of linear estimators of $\beta$ as follows:
$\{\hat{\beta}_m = \Psi_m y, m \in \mathcal{M}\}$ with $\Psi_m := \beta_m (X\beta_m)^T\Big( (X \beta_m)(X\beta_m)^T+  R R^T \Big)^{\dag}$ for all $ m \in \mathcal{M}$, and use the model selection procedure that we present
in the sequel to select the best one i.e., to recognize the actual object. \\
We should point out that the same approach might be used also in estimation;
actually when one has a-priori a finite collection of profiles of the vector of the coefficients
of the solution with respect to some orthonormal basis (or a family of orthonormal bases).
These profiles may be obtained, for example, if one knows that the solution belongs to some class of
$p-$dimensional vectors (assumed to be some discretization of continuous functions)
such that the vector of the coefficients of each vector $u$ of this class with respect
to some $p$ by $p$ orthogonal matrix $\Phi_{m'}$ can be bounded sharply by some parametric
function (a $p-$dimensional vector actually) $f_{m'}(\mu(u))$ which parameter $\mu(u)$ should
be of much lower dimension than $p$ of course (typically $2$ or $3$). In this case, one may proceed
by instantiating a finite collection of instances of the parameter $\mu(u)$ as follows
$\{\mu_m, m \in \mathcal{M}\}$, and constructs a finite family of linear estimators of $\beta$ as
follows $\{\hat{\beta}_m := \Psi_m, m \in \mathcal{M}\}$, with $\Psi_m := \beta_m (X\beta_m)^T\Big( (X \beta_m)(X\beta_m)^T+  R R^T \Big)^{\dag}$ and $\beta_m = \Phi_{m'}^T f(\mu_m)$ for all $m \in \mathcal{M}$.

\section{Data-driven model selection for linear regression and linear inverse problems}
\label{sec2}

Keeping in mind the linear regression model (\ref{model}) and the collection of linear estimators
of $\beta$ introduced in (\ref{lin_estimator}), we suggest to devise a data-driven  model
selection criterion that can select a good estimator of $\beta$ among such a
collection of its linear estimators, and we measure the performance
of any estimator $\hat{\beta}$ of $\beta$ with its quadratic risk given
by $\mathbb{E} \big [ \| \hat{\beta} - {\beta} \|^2 \big]$.

Clearly, when matrix $X$ is rank deficient, i.e., its rank is inferior than $p$,
model (\ref{model}) is not identifiable. In this case, one needs to put some assumptions
on the solution $\beta$ to guarantee identifiability of the model and to make of its
statistical inversion a possible question.Therefore, our identifiability hypothesis in this body of
work consists in saying that one knows a-priori some linear operator $\mathcal{K}$--called a noiseless
reconstructor of $\beta$-- such that one can uniquely recover the unknown vector quantity of interest $\beta$
from one noiseless observation of model (\ref{model}) (i.e., if matrix $R$ in (\ref{model}) was zero) by using a
linear formula of the form
\begin{equation}
\label{identif}
\beta = \mathcal{K} X \beta
\end{equation}
This shall be referred as the {\it linear identifiability condition}. The motivations for imposing an identifiability
condition of type (\ref{identif}) on model (\ref{model}) are presented in details in section \ref{sec3}. Please
note that without any prior on $\beta$, one can still write $\beta = \mathcal{K} X \beta + \overline{\mathcal{K}} \beta$,
with  $\mathcal{K} = X^{\dag}$ and $\overline{\mathcal{K}} = I_p - X^{\dag} X$. In this case, only the
component of $\beta$ that belongs to the subspace  $\mathcal{S} = \Big\{ X^{\dag} X t, t \in \mathbf{R}^p \Big \}$,
in other words $\mu = X^{\dag} X \beta$ (which is also the solution of $\|\mu\|^2 \rightarrow \min_{\mu \in \mathcal{S}}$)
could be theoretically recovered, and the left part of $\beta$ i.e., $\bar{\mu} = (I_p - X^{\dag}X) \beta $ which belongs
to the subspace $\mathbf{R}^{p} - \mathcal{S}$ hence is lost. Though it is generally of little use in practice,
one would recover such a least-norm estimate of the solution of (\ref{model_noiseless}) by writing the new
model $y = X\mu + Rz$, and one has $\mu$ which obeys the linear identifiability condition
with $\mathcal{K} = X^{\dag}$ since one can write $\mu = X^{\dag} X \mu$ (see also section \ref{sec3} for more details).
In section \ref{sec3}, we shall describe some situations which can lead to derive useful expressions of the noiseless
reconstructor $\mathcal{K}$ of $\beta$ to allow recovery of solutions other than such a least-norm solution.

With this being said, let us now focus on the model selection problem we stated
previously in section \ref{sec1}, and let us start by deriving for all $m \in \mathcal{M}$ the
expression of the quadratic risk of an estimator $\hat{\beta}_m$ of $\beta$, for all $m\in \mathcal{M}$. So,
one fixes some $m\in \mathcal{M}$, and from (\ref{model}) and (\ref{lin_estimator}) one has
\begin{equation}
\label{estimator_ter}
\hat{\beta}_{m} = \Psi_m y  = \Psi_m X \beta +  \Psi_m R z \nonumber
\end{equation}
hence
\begin{equation}
\label{dist_estim}
\hat{\beta}_{m} - \beta =  \big(\Psi_m X - I \big) \beta +  \Psi_m R z \nonumber
\end{equation}
then, by rising both sides of the latter equality to power of two, one obtains
\begin{multline}
\label{estim_square_dif}
\| \hat{\beta}_{m} - {\beta} \|^2  = \beta^T \big(\Psi_m X - I_p\big)^T \big(\Psi_m X - I_p\big)  \beta + z^T R^T \Psi_m^T \Psi_m R z \\+ 2   \beta^T\big(\Psi_m X - I \big)^T \Psi_m R z \nonumber
\end{multline}
Finally, by applying the expectation operator on both sides of the latter equality, one derives the following expression of
the quadratic risk of $\hat{\beta}_m$
\begin{equation}
\label{risk}
\mathbb{E}\Big[\| \hat{\beta}_{m} - {\beta} \|^2\Big]  = \beta^T \big(\Psi_m X - I_p\big)^T \big(\Psi_m X - I_p\big)  \beta + tr\Big( R^T \Psi_m^T \Psi_m R\Big) \nonumber
\end{equation}
or equivalently
\begin{equation}
\label{risk_bis2}
\mathbb{E}\Big[\| \hat{\beta}_{m} - {\beta} \|^2\Big]  =   \|\beta\|^2 + \beta^T  X^T \Psi_m^T \Psi_m X \beta - 2 \beta^T \Psi_m X \beta + \|\Psi_m R\|^2
\end{equation}
Formula (\ref{risk_bis2}) says that the quadratic risk of $\hat{\beta}_m$ decomposes as the sum of two terms; a bias term:
\begin{equation}
\label{bias_term}
\|\beta\|^2 + \beta^T  X^T \Psi_m^T \Psi_m X \beta - 2 \beta^T \Psi_m X \beta \nonumber
 \end{equation}
 and a variance term:
 \begin{equation}
 \label{variance_term} \nonumber
 \|\Psi_m R\|^2
 \end{equation}
and the best estimator of $\beta$ denoted by $\hat{\beta}_{m*}$ for some $m^* \in \mathcal{M}$ is definitely the one that achieves
the best bias-variance tradeoff in the formula of the quadratic risk (\ref{risk_bis2}); or equivalently (since $\|\beta\|^2$ is a constant),
the one that minimizes over $\mathcal{M}$ the expression
\begin{equation}
\label{oracle}
\beta^T  X^T \Psi_m^T \Psi_m X \beta - 2 \beta^T \Psi_m X \beta + \|\Psi_m R\|^2
\end{equation}
Unfortunately, since expression (\ref{oracle}) that one ideally  would like to minimize
over $\mathcal{M}$ depends on the unknown $\beta$ itself, hence such a minimization cannot
be carried out in practice. Thus, we refer to the optimal but the unaccessible procedure that minimizes
(\ref{oracle}) over $\mathcal{M}$ as an oracle, and the latter shall serve as a benchmark for assessing an
estimator's performance. \\ To the impossibility of minimizing exactly the expression of the quadratic risk adds, unfortunately,
another more challenging difficulty which has to do directly with identifiability of model (\ref{model})\footnote{Obviously,
as explained earlier in this paper, when $X\beta$ is the quantity which one looks to estimate (i.e., in the regression case),
model identifiability is not generally a concern.}. Indeed, when the rank of matrix $X$ is inferior than the size of the solution $\beta$ ($p$), only the linear transformation $X\beta$ of $\beta$ can be observed (up to noise), and because of the term $\beta^T \Psi_m X \beta$ in (\ref{oracle}) which doesn't depends upon $\beta$ only through $X\beta$, this makes it quite challenging
to derive a satisfactory data-driven model selection procedure. Nonetheless, with the identifiability
condition that we introduced earlier in this section, such a model selection procedure becomes possible. Indeed,
assume that one can write $\beta = \mathcal{K} X \beta$, one deduces the following new formula of the quadratic risk
\begin{equation}
\label{risk_id}
\mathbb{E}\Big[\| \hat{\beta}_{m} - {\beta} \|^2\Big]  =   \|\beta\|^2 + (X\beta)^T \Psi_m^T \Psi_m X \beta - 2 (X\beta)^T \mathcal{K}^T \Psi_m X \beta + \|\Psi_m R\|^2 \nonumber
\end{equation}
hence formula (\ref{oracle}) that one would like to minimize over $\mathcal{M}$ becomes
\begin{equation}
\label{oracle_id_c}
 (X\beta)^T \Big(\Psi_m^T \Psi_m- 2 \mathcal{K}^T\Psi_m \Big)  X \beta + \|\Psi_m R\|^2 \nonumber
\end{equation}
The latter formula has the advantage to depend on $\beta$ only through $X\beta$ which,
as we shall see in the remainder, makes it now possible to derive useful data-driven model
selection procedures.

\subsection{Model selection}
Now, we  propose to select an estimator of $\beta$ among its finite set of linear estimators
(\ref{lin_estimator}) by minimizing a penalized criterion of the form
\begin{equation}
\label{crit_quad}
\textrm{Crit}(m) = y^T \Big ( \Psi_m^T \Psi_m  - 2 \mathcal{K}^T  \Psi_m) y + \textrm{pen}(m)
\end{equation}
where $\textrm{pen}(m)$ stands for an arbitrary penalty function which good choice shall be addressed
later in this section. The estimator of $\beta$ denoted by $\tilde{\beta} = \hat{\beta}_{\tilde{m}}$
that minimizes criterion (\ref{crit_quad}) shall be referred as the penalized estimator of $\beta$,
and its quadratic risk $\mathbb{E}\Big[\| \tilde{\beta} - {\beta} \|^2\Big] $ shall be
referred as the penalized risk. The following theorem provides then some useful clues that help
choose the penalty $\textrm{pen}(m)$ in (\ref{crit_quad}).

\begin{theorem}
\label{main_result_q}
Assume model (\ref{model}) along with the collection of linear estimators of
$\beta$ introduced in (\ref{lin_estimator}), and assume equality (\ref{identif}).
Let us fix some positive number $\theta \in (0,1)$, and  consider for all $m \in \mathcal{M}$ the matrices
\begin{equation}
\mathcal{A}_m = R^T \Big (-\theta \Psi_m^T \Psi_m + \mathcal{K}^T  \Psi_m + \Psi_m^T  \mathcal{K} \Big ) R \nonumber
\end{equation}
\begin{equation}
\mathcal{B}_m = \Big (\theta \Psi_{m} - \mathcal{K}\Big) R R^T \Big (\theta \Psi_{m} - \mathcal{K}\Big)^T \nonumber
\end{equation}
 and denote by $s_m^{+}$ the largest positive eigen value of the symmetric matrix $\mathcal{A}_m$ and  by $r_m^{*}$ the largest
singular value of matrix $\mathcal{B}_m$. Let us now associate to each model $m \in \mathcal{M}$ two real positive numbers  $L_m$ and $h_m$ with the $L_m$'s satisfying $\Sigma = \sum_{m \in \mathcal{M}} \exp[- L_m ] < \infty$, and consider for all $m \in \mathcal{M}$
the quantity $Q_m$ defined as follows
\begin{multline}
Q_m = 2 tr\Big(R^T \mathcal{K}^T  \Psi_{m}R\Big)  -\theta \|\Psi_m R\|^2 + 2\big(\frac{r_{m}^*}{\theta} + s_{m} ^+\big) L_m + 2 \|A_{m} \|\sqrt{L_m+ h_m}\\
+ \lambda \Big ( \frac{\|A_{m} \|}{\sqrt{L_m} + \sqrt{L_m + h_m}} + \frac{r_{m}^*}{\theta} + s_{m}^+ \Big )^2 \nonumber
\end{multline}
for some positive number $\lambda$. It follows that for every penalty
function $\textrm{pen}(m), m \in \mathcal{M}$, the corresponding  penalized
estimator $\tilde{\beta}$ satisfies
\begin{multline}
(1-\theta)\mathbb{E}\Big[\|\tilde{\beta} - {\beta} \|^2 \Big] \leq  \inf_{m \in \mathcal{M}} \Big\{ \|\beta\|^2 +  \beta^T  X^T \Psi_m^T \Psi_m X \beta - 2 \beta^T \Psi_m X \beta + \|\Psi_m R\|^2 \\
- 2 tr\Big(R^T \mathcal{K}^T \Psi_m R \Big)  +  \textrm{pen}(m) \Big\} + \sup_{m \in \mathcal{M}} \Big ( Q_m -  \textrm{pen}(m) \Big ) + \frac{2\Sigma}{\lambda} \nonumber
\end{multline}
\end{theorem}

The proof of this theorem is deferred to the appendix section. \\

\subsection{Equivalent of the model selection result in the regression framework}
\label{sub_eq}
Though we considered the inverse problem framework and we argued that this was
w.l.g., however before commenting in details theorem \ref{main_result_q},
we would like to give its equivalent in the regression framework as this might
help make our approach as much clear as possible. Another motivation that we would like
to emphasize and which we had already the opportunity to discuss in the beginning
of this section is that, in contrast to the inverse problem framework, the
identifiability assumption about model (\ref{model}) is not necessary in the regression
framework to use the proposed model selection approach.
Indeed, such a linear identifiability condition might turn to be difficult
to establish in some applications, however, one might use instead
some priors about the solution (see the examples in section \ref{sec1})
so as to enhance identification of the solution, and feel rather
safe to use the  predictive risk as an estimator's measure
of performance.

Now, consider the finite collection of linear estimators of $\beta$ given in (\ref{lin_estimator}),
then one establishes the following formula of the predictive risk of $\hat{\beta}_{m}$,
for all $m\in \mathcal{M}$
\begin{equation}
\mathbb{E}\Big[X\|\hat{\beta}_{m} - X {\beta}_{m} \|^2\Big]  = \|X\beta\|^2 +  (X\beta)^T  \Big(\Psi_m^T X^T X \Psi_m -2 X \Psi_m \Big) (X\beta) \\
+ \|X \Psi_m R\|^2 \nonumber
\end{equation}
and one selects the model of $\beta$ by minimizing over $\mathcal{M}$
a penalized criterion of the form
\begin{equation}
\label{crit_p}
\textrm{Crit}_{pred}(m) = y^T \big(\Psi_m^T X^T X \Psi_m - 2X \Psi_m \big) y   + \textrm{pen}_{pred}(m)
\end{equation}
for an arbitrary penalty $\textrm{pen}^{pred}(m), m\in \mathcal{M}$. Let us
denote by $\tilde{\beta}^{pred} = \hat{\beta}_{\tilde{m}}$ the estimator
of $\beta $ that minimizes (\ref{crit_p}) over $\mathcal{M}$, then the
following theorem provides an upper bound of the predictive risk of 
$\tilde{\beta}^{pred}$.
\begin{theorem}
\label{main_result_p}
Assume the linear model (\ref{model}), along with the finite collection
linear estimators of $\beta$ introduced in (\ref{lin_estimator}). Let us fix
some positive number $\theta \in (0,1)$, and  consider  for all $m \in \mathcal{M}$
the matrices
\begin{equation}
\nonumber
\mathcal{A}_m^{pred} =  R^T \big(-\theta \Psi_{m} ^T X^T X  \Psi_{m} + X \Psi_{m}  + \Psi_{m}^T X^T \big)  R  \nonumber
\end{equation}
\begin{equation}
\nonumber
\mathcal{B}_m^{pred} = \big(\theta X\Psi_{m}  - I_n\big) R R^T \big(\theta X \Psi_{m}  - I_n\big)^T  \nonumber
\end{equation}
and denote by $s_m^{+}$ the largest positive eigen value of the symmetric matrix $\mathcal{A}_m^{pred}$ and  by $r_m^{*}$ the largest  singular value of matrix $\mathcal{B}_m^{pred}$. Let us now associate to each model $m \in \mathcal{M}$ two real positive numbers  $L_m^{pred}$ and $h_m^{pred}$ with the $L_m^{pred}$'s satisfying $\Sigma^{pred} = \sum_{m \in \mathcal{M}} \exp[- L_m^{pred} ] < \infty$, and consider for $m \in \mathcal{M}$ the quantity
{\begin{multline}
Q_m^{pred} = 2 tr\Big(R^T X\Psi_{m}R\Big)  -\theta \|X \Psi_m R\|^2 + 2\big(\frac{r_{m}^*}{\theta} + s_{m} ^+\big) L_{m}^{pred} + 2 \|A_{m} \|\sqrt{L_{m}^{pred}+ h_m^{pred}}\\ + \lambda \Big ( \frac{\|A_{m} \|}{\sqrt{L_{m}^{pred}} +
 \sqrt{L_{m}^{pred} + h_m^{pred}}} + \frac{r_{m}^*}{\theta} + s_{m}^+ \Big )^2 \nonumber
\end{multline}}
for some positive number $\lambda$. It then follows that for every  penalty function $\textrm{pen}_{pred}(m)$ for all $m \in \mathcal{M}$, the selected estimator $\tilde{\beta}^{pred}$ satisfies
\begin{multline}
(1-\theta)\mathbb{E}\Big[\|X\tilde{\beta}^{pred} - {X\beta} \|^2 \Big]
\leq  \inf_{m \in \mathcal{M}} \Big \{ \| X\beta\|^2 + (X\beta)^T \Big( \Psi_m^T X^T X \Psi_{m} - 2 X \Psi_{m} \Big  ) (X\beta)  \\
+  \|X \Psi_m R \|^2 - 2 tr\Big( R^T X \Psi_{m}R\Big)+ \textrm{pen}_{pred}(m) \Big \} + \sup_{m \in \mathcal{M}} \Big \{ Q_m^{pred} -  \textrm{pen}_{pred}(m) \Big \} + \frac{2\Sigma^{pred}}{\lambda}
\nonumber
\end{multline}
\end{theorem}
\begin{proof}
Theorem \ref{main_result_p} is a simple consequence of theorem \ref{main_result_q}
by taking into account the remarks above about the fact that
considering the quadratic risk as a performance measure of an
estimator $\widehat{X\beta}$ of $X\beta$ amounts to considering
the predictive risk as a performance measure of an estimator
$\hat{\beta}$ of $\beta$.
\end{proof}
As such a model selection result in the case of the predictive risk
is a special case of the more general result in the inverse problem
framework, therefore all our comments below, except the identifiability issue,
apply to such a regression framework as well.

\subsection{Commenting the choice of the penalty}
\label{sub_pen}
Theorem \ref{main_result_q} suggests to take the penalty function,
at least for models $m$ for which the quantity $Q_m$ is very large,
such that $\textrm{pen}(m) \geq  Q_m$ and choose the $L_m$'s in such a way
to get the value of the quantity $\frac{2\Sigma}{\lambda}$ very small
in order to achieve a reasonable value of the upper bound of the penalized
quadratic risk. Hence, in order to see the performance of the proposed model
selection approach, one has to choose accordingly the penalty function and
its constants, and let us assume in the remainder that the solution $\beta$
obeys the linear identifiability condition.

Thus, we propose to choose the penalty function for all $m \in \mathcal{M}$ such that $\textrm{pen}(m) =  Q_m$,
in other words, one puts  for all $m \in \mathcal{M}$
{\begin{multline}
\label{penalty}
\textrm{pen}(m) := 2 tr\Big(R^T \mathcal{K}^T  \Psi_{m}R\Big)  -\theta \|\Psi_m R\|^2
+ 2\big(\frac{r_{m}^*}{\theta} + s_{m} ^+\big) L_m  \\
+ 2 \|A_{m} \|\sqrt{L_m+ h_m} + \lambda \Big ( \frac{\|A_{m} \|}{\sqrt{L_m} +
\sqrt{L_m + h_m}} + \frac{r_{m}^*}{\theta} + s_{m}^+ \Big )^2
\end{multline}}
Theorem  \ref{main_result_q} then guarantees that the penalized estimator $\tilde{\beta}$ obtained with such a choice of the penalty satisfies
\begin{multline}
\nonumber
(1-\theta)\mathbb{E}\Big[\|\tilde{\beta} - {\beta} \|^2 \Big] \leq  \inf_{m \in \mathcal{M}} \Big\{ \|\beta\|^2 +  \beta^T  X^T \Psi_m^T \Psi_m X \beta - 2 \beta^T \Psi_m X \beta \\ + (1-\theta)\|\Psi_m R\|^2
+ 2\big(\frac{r_{m}^*}{\theta} + s_{m} ^+\big) L_m + 2 \|A_{m} \|\sqrt{L_m+ h_m} \\ + \lambda \Big ( \frac{\|A_{m} \|}{\sqrt{L_m} + \sqrt{L_m + h_m}} + \frac{r_{m}^*}{\theta} + s_{m}^+ \Big )^2  \Big \} + \frac{2\Sigma}{\lambda}
\end{multline}

However, it is still difficult to see in the latter formula how the performance of the proposed
penalized estimator compares with the lowest value of the predictive risk over $\mathcal{M}$
which is attained only by an oracle of course. Hence, the following proposition provides
an interesting strategy of choice of the constants in the penalty function
(\ref{penalty}) which enforces an oracle inequality.
\begin{proposition}
\label{prop2}
Assume a penalty of the form (\ref{penalty}). It follows that if one defines for all $m \in \mathcal{M}$
the $L_m$'s, for some positive sequence $\ell_m, m \in \mathcal{M}$ (these shall be referred as the model weights)
and for some $\alpha\in (0,1]$, as follows
\begin{equation}
L_m := \frac{\|\Psi_m R \|^6 \ell_m }{ (\frac{2r_{m}^*}{\theta}+ s_{m}^+)  \| \Psi_m R \|^4 + 4 \alpha^2 \|A_m\|^2 \big( \|\Psi_m R \|^2+ h_m\big)} \nonumber
\end{equation}
and defines, for some positive number $\epsilon$ the $h_m$'s as follows
\begin{equation}
\nonumber
h_m := \frac{\|A_m\|^2}{\epsilon^2(\frac{2r_{m}^*}{\theta}+ s_{m}^+)^2} \mbox{1{\hskip -2.5 pt}\hbox{I}}_{\frac{\|A_m\|}{2\sqrt{t_m}}\geq \epsilon (\frac{2r_{m}^*}{\theta}+ s_{m}^+)}
\end{equation}
with
\begin{equation}
t_m := \frac{\|\Psi_m R \|^6 \ell_m }{ (\frac{2r_{m}^*}{\theta}+ s_{m}^+)  \| \Psi_m R \|^4 + 4 \alpha^2 \|A_m\|^2  \|\Psi_m R \|^2} \nonumber
\end{equation}
and takes
\begin{equation}
\nonumber
 \lambda :=  \frac{\sqrt{2}\Sigma^{\frac{1}{2}}}{\sup_{m\in \mathcal{M}} \Big \{ \Big ( \frac{\|A_{m} \|}{2\sqrt{L_m}} + \frac{r_{m}^*}{\theta} + s_{m}^+ \Big ) \Big \}}
\end{equation}
then the penalized estimator $\tilde{\beta}$ satisfies
\begin{multline}
\nonumber
\mathbb{E}\Big[\|\tilde{\beta} - {\beta} \|^2 \Big] \leq \frac{1}{1-\theta} \inf_{m \in \mathcal{M}} \Big \{\|\beta\|^2 +  \beta^T  X^T \Psi_m^T \Psi_m X \beta - 2 \beta^T \Psi_m X \beta \\
\Big( 1 - \theta + \ell_m + \frac{1}{\alpha}\sqrt{\ell_m}  \Big)\|\Psi_m R \|^2 \Big \}  +
 \frac{2 \sqrt{2} (1+\epsilon )\Sigma^{\frac{1}{2}}}{1-\theta}  \Big[ \sup_{m\in \mathcal{M}} \Big\{ \Big(\frac{r_{m}^*}{\theta} + s_{m}^+  \Big ) \Big \} \Big]
\end{multline}
\end{proposition}

The proof of proposition \ref{prop2} shall be given in the form
of a commentary of theorem \ref{main_result_q} following
our choice of the penalty (\ref{penalty}). \\
So, let us assume a penalty of the form (\ref{penalty}), and choose the $L_m$'s
for some positive sequence of model weights $\ell_m, m \in \mathcal{M}$ and for
some positive number $\alpha$  (one can assume $\alpha\in (0,1]$ w.l.g.),  as follows
\begin{equation}
L_m := \frac{\|\Psi_m R \|^6 \ell_m }{ (\frac{2r_{m}^*}{\theta}+ s_{m}^+)  \| \Psi_m R \|^4 + 4 \alpha^2 \|A_m\|^2 \big( \|\Psi_m R \|^2+ h_m\big)} \nonumber
\end{equation}
and by noticing that
\begin{equation}
2\big(\frac{r_{m}^*}{\theta} + s_{m} ^+\big) L_m \leq \ell_m \| \Psi_m R \|^2 \nonumber
 \end{equation}
and
\begin{equation}
2 \|A_{m} \|\sqrt{L_m} \leq \frac{1}{\alpha}\sqrt{\ell_m} \| \Psi_m R \|^2 \nonumber
\end{equation}
one obtains the new upper bound of the penalized risk
\begin{multline}
(1-\theta)\mathbb{E}\Big[\|\tilde{\beta} - {\beta} \|^2 \Big]
\leq  \inf_{m \in \mathcal{M}} \Big \{ \|\beta\|^2 +  \beta^T  X^T \Psi_m^T \Psi_m X \beta - 2 \beta^T \Psi_m X \beta + \\
 \Big( 1 - \theta +  \ell_m + \frac{1}{\alpha}\sqrt{\ell_m} \Big) \|\Psi_m R \|^2 + \lambda \Big ( \frac{\|A_{m} \|}{\sqrt{L_m} + \sqrt{L_m + h_m}} + \frac{r_{m}^*}{\theta} + s_{m}^+ \Big )^2 \bigg \} + \frac{2\Sigma}{\lambda} \nonumber
\end{multline}
Now, if one puts
\begin{equation}
\nonumber
 \lambda :=  \frac{\sqrt{2}\Sigma^{\frac{1}{2}}}{\sup_{m\in \mathcal{M}} \Big \{ \Big ( \frac{\|A_{m} \|}{2\sqrt{L_m}} + \frac{r_{m}^*}{\theta} + s_{m}^+ \Big ) \Big \}}
\end{equation}
one finds that $\tilde{\beta}$ satisfies
\begin{multline}
\label{upper_bound_p5}
\mathbb{E}\Big[\|\tilde{\beta} - {\beta} \|^2 \Big] \leq \frac{1}{1-\theta} \inf_{m \in \mathcal{M}} \Big \{\|\beta\|^2 +  \beta^T  X^T \Psi_m^T \Psi_m X \beta - 2 \beta^T \Psi_m X \beta \\
\Big( 1 - \theta + \ell_m + \frac{1}{\alpha}\sqrt{\ell_m}  \Big)\|\Psi_m R \|^2 \Big \}  + \frac{2 \sqrt{2}}{1-\theta} \Big [ \sup_{m\in \mathcal{M}} \Big\{ \Big( \frac{\|A_{m} \|}{2\sqrt{L_m}} + \frac{r_{m}^*}{\theta} + s_{m}^+
 \Big )^2 \Big \} \Sigma \Big]^{\frac{1}{2}}
\end{multline}
which recalls the oracle equality (\ref{oracle}) up to the per-model multiplicative
quantity $\Big( 1 - \theta + \ell_m + \frac{1}{\alpha}\sqrt{\ell_m}  \Big)$ and the additive
constant which is given by
\begin{equation}
\nonumber
\Gamma(\theta, \alpha, \mathcal{L}) :=  \frac{2 \sqrt{2}}{1-\theta} \Big [ \sup_{m\in \mathcal{M}} \Big\{ \Big( \frac{\|A_{m} \|}{2\sqrt{L_m}} + \frac{r_{m}^*}{\theta} + s_{m}^+  \Big )^2 \Big \} \Sigma \Big]^{\frac{1}{2}}
\end{equation}
and which can be made universal at the price of augmenting the $\ell_m$'s. It is interesting
to note that one can see the quantity $\Big( 1 - \theta + \ell_m + \frac{1}{\alpha}\sqrt{\ell_m} \Big) \|\Psi_m R \|^2 $
as the difficulty of representing the solution $\beta$ by using model $m$,
expressed in different words, one would say that such a quantity measures somehow the
complexity of a given model $m$ with respect to the model collection  $\mathcal{M}$.

Regarding now the choice of the model weights $\ell_m$, one has to distinguish
generally  between the two cases : a small collection of models and a large collection of models.
To start, if one assumes a small collection of models (say, a few dozens), then one may take them to be fix,
i.e., $\ell_m := \ell$ for all $m \in \mathcal{M}$ while guaranteeing that for a reasonable choice of the constant $\ell$,
the value of $\Sigma$ can be bounded by a small enough constant $C$. In this case, one obtains over the set
 $\mathscr{L}$ of all collections of estimators such that $\mathscr{L} = \big\{\mathcal{L}', \textrm{ s.t. } \Gamma(\theta,\alpha, \mathcal{L}') \leq C \big\}$  a near-oracle performance of the form
\begin{equation}
\nonumber
 \mathbb{E}\Big[\|\tilde{\beta} - {\beta} \|^2 \Big] \leq K \inf_{m \in \mathcal{M}}  \mathbb{E}\Big[\|\hat{\beta}_m - {\beta} \|^2 \Big] + C
\end{equation}
with $ K = K(\theta, \alpha, C) = \frac{1- \theta +\ell(C) + \frac{1}{\alpha }\sqrt{\ell(C)} }{1-\theta}$.

Now, assume a large family of models, then taking the $\ell_m$'s as a fixed value for all $m \in \mathcal{M}$
may result in an inflation of the upper bound of the penalized risk. To account for this situation,
let us first define the following quantity for all $m \in \mathcal{M}$
\begin{equation}
\Delta_m := \frac{\|\Psi_m R \|^6 \ell_m }{ (\frac{2r_{m}^*}{\theta}+ s_{m}^+)  \| \Psi_m R \|^4 + 4 \alpha^2 \|A_m\|^2 ( \|\Psi_m R \|^2+ h_m)} \nonumber
\end{equation}
which can  be seen, in some sense, as a measure of a model's complexity, hence one can write $L_m =\Delta_m \ell_m$.
To start in exhibiting the intuition for the subsequent choice of the $\ell_m$'s, let us fix ideas
by assuming, for the moment, that  $\Delta_m$ can take its values in an ordered set of discrete values, and for
simplicity, one can assume that such values are the integers in the interval $\big[\Delta_{min}, \Delta_{max}\big]$. Let us also denote by $\textrm{card}_{\mathcal{M}}(\Delta) = \sum_{m\in \mathcal{M}} \mbox{1{\hskip -2.5 pt}\hbox{I}}_{\Delta_m = \Delta}$,
and put $L_m := L_{\Delta_m}$, for all $m\in \mathcal{M}$ which makes sense since one would like to
roughly privilege among models with the same complexity the one(s) which make(s) a difference
 (i.e., achieves the lowest value) in the bias term of the penalized risk. Hence one finds that
\begin{eqnarray}
\Sigma & = &\sum_{m \in \mathcal{M}} \exp\big[-\ell_m \Delta_m\big] =  \sum_{\Delta = \Delta_{min}}^{\Delta_{max}} \textrm{card}_{\mathcal{M}}(\Delta) \exp\big[-\ell_{\Delta} \Delta\big] \nonumber \\
& = & \sum_{\Delta = \Delta_{min}}^{\Delta_{max}} \exp\Big [\log\big[\textrm{card}_{\mathcal{M}}(\Delta)\big] -\ell_{\Delta} \Delta\Big] \nonumber
\end{eqnarray}
then, if one takes $\ell_{\Delta} = \delta + \frac{\log\big[\textrm{card}_{\mathcal{M}}(\Delta)\big]}{\Delta}$, interestingly,
one finds that $\Sigma  \leq \frac{\exp\big[-\Delta_{min}\delta \big]}{1-\exp[-\delta]}$; which means that with
such a choice of the model weights, one can make arbitrarily small the value of $\Sigma$. \\
Therefore, one would like to generalize this idea to the actual case in which the $\Delta_m$'s
do not take discrete values but continuous values instead in the hope of bounding sharply the
value of $\Sigma$ while keeping the expression of the penalized risk reasonable. Then one way to
achieve such a goal would be by mimicking the discrete quantity $\textrm{card}_{\mathcal{M}}(\Delta_m)$ in $\mathbf{R}^+$
by the means of a kernel diffusion for instance\footnote{This is remindful of Kernel Density Estimation (KDE), the
difference here is that the $\Delta_m$'s are deterministic values \dots}. By choosing the latter to be
a gaussian kernel with a bandwidth $\sigma$, then one possible expression that
mimics interestingly $\textrm{card}_{\mathcal{M}}(\Delta_m)$ in $\mathbf{R}^+$ is
\begin{equation}
\nonumber
 \sum_{m' \in \mathcal{M}} \exp\Big [ - \frac{\big(\Delta_{m'} - \Delta_{m}\big)^2}{2\sigma^2} \Big]
\end{equation}
One checks that in the discrete case, the latter quantity converges to  $\textrm{card}_{\mathcal{M}}(\Delta_m)$  as $\sigma^2 \rightarrow 0$. One may then take the $\ell_{m}$'s for some positive number $\delta$ as follows
\begin{equation}
\nonumber
\ell_{m} := \delta + \frac{\log\Big[{ \sum_{m' \in \mathcal{M}} \exp\big [ - \frac{\big(\Delta_{m'} - \Delta_{m}\big)^2}{2\sigma^2} \big]}\Big]}{\Delta_m}
\end{equation}
in which case, one finds that
\begin{equation}
\nonumber
\Sigma = \sum_{m \in \mathcal{M}} \frac{ \exp\big[-\delta \Delta_m\big]}{\sum_{m' \in \mathcal{M}} \exp\Big [ - \frac{\big(\Delta_{m'} - \Delta_{m}\big)^2}{2\sigma^2}\Big]}
\end{equation}
Though it is not generally possible to give a sharp closed-form expression of $\Sigma$,
one has $\Sigma$ which is overwhelmingly bounded by $\sum_{m \in \mathcal{M}}
\exp\big[-\delta \Delta_m\big]$ for a reasonable choice of the constant $\sigma^2$ (i.e.,
not taking it too small), hence if one fixes the value of $\sigma^2$,
a proper setting of the constant $\delta$ allows to bring down enough the value of
$\Gamma(\theta, \alpha, \mathcal{L})$, i.e., in such a way to achieve $\Gamma(\theta, \alpha, \mathcal{L}) \leq C $
for some small enough constant $C$. One should choose the constant $C$ to be small enough so as not
to inflate the additive constant $\Gamma(\theta, \alpha, \mathcal{L})$ in the formula of the upper
bound of the penalized risk, and big enough in order to keep the model weights $\ell_m, m \in \mathcal{M}$ reasonable.
In practical applications, one should take it as follows $C := o(\|\beta\|^2)$.\\
Concerning the choice of $\sigma$, there is not an optimal value of it actually since as $\sigma$ increases,
the $\ell_m$'s decrease (desired) resulting in the increase of $\Sigma$ (undesired), and vice-versa. Nevertheless,
there exists some intuitive facts that suggest to take $\sigma$  as follows
\begin{equation}
\sigma := \frac{\tau}{\mu p} \nonumber
\end{equation}
with $\mu$ which can be taken for example as follows $\mu:= 3$, and $\tau$ which is some positive
quantity which represents in some sense the scale of the quantities $\Delta_m, m \in \mathcal{M}$,
which, therefore, may be taken as follows
\begin{equation}
\nonumber
\tau := \sqrt{\frac{1}{|\mathcal{M}|} \sum_{m, m' \in \mathcal{M}} \Big(\Delta_m - \Delta_{m'}\Big)^2}
\end{equation}
In fact, the intuition for such a choice of $\sigma$ is that, if one had to redefine discrete
values for $\Delta_m$, one would achieve this as follows $\Delta_m := [\frac{\Delta_m}{\tau}] \cdot \tau, \forall m \in \mathcal{M}$ with $[x]$ which stands for the integer value of $x$. Hence by mapping the $|\mathcal{M}|$ values of $\Delta_m$ into the discrete interval $[1,\cdots,p]$, one realizes in some way a discretization of the $\Delta_m$'s with a discretization step which is equal to $\frac{\tau}{p}$, hence one retrieves somehow the discrete case that we discussed above. It follows that imitating the quantity $\textrm{card}_{\mathcal{M}}(\Delta)$ amounts to taking $3 \sigma \approx \frac{\tau}{p}$ (since almost $90\%$ of the energy of a gaussian distribution with standard deviation $\sigma$ is concentrated in an interval of size $\mu\sigma$ around its mean, with $\mu\approx 3$). \\

Regarding the setting of the $h_m$'s, following our choice above of the penalty and some of
its constants namely the $L_m$'s and $\lambda$, their role principally is to prevent the inflation of the additive constant $\Gamma(\theta,\alpha, \mathcal{L})$ thereby the inflation of the upper bound of the penalized risk when there are among the
collection $\mathcal{M}$ some models with too small weights (which can be seen
as nuisance models in some sense). So, to address the choice $h_m$'s, first let us denote by
\begin{equation}
t_m := \frac{\|\Psi_m R \|^6 \ell_m }{ (\frac{2r_{m}^*}{\theta}+ s_{m}^+)  \| \Psi_m R \|^4 + 4 \alpha^2 \|A_m\|^2  \|\Psi_m R \|^2} \nonumber
\end{equation}
hence, one may define the $h_m$'s as follows
\begin{equation}
\nonumber
h_m := \frac{\|A_m\|^2}{\epsilon^2(\frac{2r_{m}^*}{\theta}+ s_{m}^+)^2} \mbox{1{\hskip -2.5 pt}\hbox{I}}_{\frac{\|A_m\|}{2\sqrt{t_m}}\geq \epsilon (\frac{2r_{m}^*}{\theta}+ s_{m}^+)}
\end{equation}
and one has
\begin{equation}
\frac{\|A_{m} \|}{\sqrt{L_m} + \sqrt{L_m + h_m}} + \frac{r_{m}^*}{\theta} + s_{m}^+ \leq
(1+\epsilon )\big( \frac{r_{m}^*}{\theta} + s_{m}^+\big) \nonumber
\end{equation}
one finds that
\begin{eqnarray}
\nonumber
\Gamma(\theta, \alpha, \mathcal{L}) & = & \frac{2 \sqrt{2} (1+\epsilon )}{1-\theta} \Big[ \sup_{m\in \mathcal{M}} \Big\{ \Big(\frac{r_{m}^*}{\theta} + s_{m}^+  \Big )^2 \Big \} \Sigma \Big]^{\frac{1}{2}} \nonumber \\
& = &  \frac{2 \sqrt{2} (1+\epsilon )\Sigma^{\frac{1}{2}}}{1-\theta}  \Big[ \sup_{m\in \mathcal{M}} \Big\{ \Big(\frac{r_{m}^*}{\theta} + s_{m}^+  \Big ) \Big\} \Big]\nonumber
\end{eqnarray}
One may choose the constant $\epsilon$ for instance as follows $\epsilon :=1$;
which turns to be a good trade-off between keeping the constant $\Gamma(\theta, \alpha, \mathcal{L})$
reasonable on the one hand, and not inflating too much the weights $\ell_m$ on the other hand.

For the choice of two remaining penalty parameters namely $\theta$
and $\alpha$, as one could notice, their optimal choice strictly
speaking doesn't exist as this suggests knowledge upfront of the best
model. Nevertheless, it is easy to choose them very reasonably without 
prior knowledge of the best model by examining the formula of the upper 
bound of the penalized quadratic risk (\ref{upper_bound_p5}).
Indeed, firstly for $\theta$,  formula (\ref{upper_bound_p5}) suggests
to use values which lie between but far enough from the critical values $0$
and $1$, hence, unless there is a possibility of optimizing the value of
$\theta$ for a specific application (e.g. by simulation), we suggest to
choose it for example as follows $\theta:=\frac{3}{4}$ which happens to be a
good compromise between the bias and the variance term in the formula of the upper
bound of the penalized risk. Secondly for $\alpha$, one can see that increasing
$\alpha$ results in the decrease of the quantity $\big( 1 - \theta +  \ell_m + \frac{1}{\alpha}\sqrt{\ell_m} \big)$
and the increase of $\Sigma$ at the same time. And by remarking that the dominant
quantity in $\ell_m + \frac{1}{\alpha}\sqrt{\ell_m}$ is $\ell_m$, hence any value
greater enough than $0$ would be adequate for $\alpha$, hence by default, one may take
it as follows  $\alpha := 1$ which happens to be a reasonable choice.

The strategy of choice of the penalty and its constants that we presented in this section
constitutes of course one possible strategy that turns to enforce an oracle inequality
of the selected estimator independently of any application, and we do not exclude that
more optimized penalties for some specific applications--by the means for instance of
a simulation study that attempts to learn the optimal values of the involved
constants in the penalty for these applications--could be designed.
Other papers dedicated more to applications will follow this one anyway
in which the dear reader can find more hints for optimizing the penalty
for some specific applications.

\section{Linear identifiability assumption}
\label{sec3}
Whenever the estimation of the vector $\beta$ from (\ref{model}) is
the statistician's main concern (i.e., the inverse problem framework),
and the rank of matrix $X$ is inferior than $p$, one is faced
unavoidably to an identifiability problem of model (\ref{model}) that needs to be addressed
before using the proposed model selection approach. To put forward such
an issue, let us consider the following noiseless linear model
\begin{equation}
\label{model_noiseless}
y' = X \beta
\end{equation}
and consider the two subspaces of $\mathbf{R}^P$:
$\mathcal{S}_1 = \big \{ X^{\dag} X t, t \in \mathbf{R}^p \big\}$, and
$\mathcal{S}_2 = \big \{ (I_p -  X^{\dag} X) t, t \in \mathbf{R}^p \big\}$.
Please note that $\mathcal{S}_1 \cup \mathcal{S}_2 = \mathbf{R}^P$, and $\mathcal{S}_1 \cap \mathcal{S}_2 = {\O}$,
one deduces that all solutions of the form
\begin{equation}
\beta(\eta)  = X^{\dag} y' + (I_p - X^{\dag} X) \eta \nonumber
\end{equation}
for any $p-$dimensional real vector $\eta$, satisfy equation (\ref{model_noiseless}).
This simply means that, unless the component of $\beta$ that belongs to
$\mathcal{S}_2$, i.e., $(I_p - X^{\dag} X) \beta$ is known or identifiable with
respect to the identifiable component $X^{\dag} X \beta$ ,
one cannot generally recover $\beta$ from an instance of (\ref{model_noiseless}),
hence its estimation from an observation of model (\ref{model}) is highly problematic.

Therefore, in order to guarantee identifiability of model (\ref{model}), the hypothesis
which we put forward in section \ref{sec3} and which we called the linear identifiability
condition consists in saying that  one knows {\it a priori} a linear operator
$\mathcal{K}$ that we called a noiseless reconstructor of $\beta$ such that one can
write $\beta = \mathcal{K} X \beta$ or equivalently $(I_p - \mathcal{K} X )\beta = 0$. Such
a linear identifiability condition happens in fact to be realistic for various practical
applications and some of them are discussed below.

Firstly, in many engineering fields, one commonly assumes that the solution
$\beta$ that one would like to recover from an observation of (\ref{model}) is compressible
in some basis $\Lambda$ of $\mathbf{R}^p$ (for instance some wavelet basis \cite{mallat_1989, meyer_1990, daubechies_1992});
which means that if one computed the coefficients of the scalar product between $\beta$ and the respective vectors
of the basis $\Lambda$, many of these coefficients would be found to be zero (or almost zero).
We show indeed in this case that under some mild assumptions which happen to be realistic
for numerous practical applications, one may derive easily a possible expression of the
noiseless reconstructor  $\mathcal{K}$ that allows to overcome the identifiability
problem of the model. In short, saying that $\beta$ is compressible in some basis $\Lambda$ means that
one is capable of extracting from $\Lambda$ a subset of vectors which scalar product with $\beta$ is
fatally zero \footnote{When for instance $\beta$ is known to be smooth, any trigonometric or wavelet basis might do the job.}.
Hence one proceeds by extracting from $\Lambda$ a number $k = p-\textrm{rank}(X)$ of
such vectors, then use them to construct the rows of a matrix $\phi$ so as to enforce
on the solution $\beta$ an equality of the form $\phi \beta = 0$. One checks easily that if matrix $\phi$ satisfies
simultaneously  the two following conditions:
\begin{enumerate}
\item $\phi \beta = 0$
\item the rank of the augmented matrix $\left [ \begin{array}{c}X \\ \phi \end{array}\right]$ --defined by
the union of the rows of $X$ and the rows of $\phi$-- is equal to $p$
\end{enumerate}
then holds necessarily the following linear relationship between the two vectors $\beta$ and $X\beta$:
\begin{equation}
\label{beta_noiseless}
\nonumber
\beta = \Big( X^T X  + \phi^T \phi \Big)^{-1} X^T X\beta
\end{equation}
Hence, one can define the linear reconstructor $\mathcal{K}$ of $\beta$ as follows
\begin{equation}
\label{noiseless_reconst}
\nonumber
\mathcal{K} = \mathcal{K}(X, \phi) := \Big( X^T X  + \phi^T \phi \Big)^{-1} X^T
\end{equation}
More generally, one shows that for any matrix $\Pi$ such that the rank of the augmented matrix
$\left [ \begin{array}{c}X \\ \Pi \end{array}\right]$ --which rows
are the union of the rows of $X$ and the rows of $\Pi$-- is equal to $p$ then one has
the following linear relationship which holds for the two vectors $\beta$ and $X\beta$:
\begin{equation}
\label{beta_noiseless_b}
\nonumber
\beta = \Big( X^T X  +  \Pi^T \Pi \Big)^{-1} X^T X \beta + \Big(\Pi \big(I_p - X^{\dag} X\big)\Big)^{\dag} \Pi \beta
\end{equation}
In particular, if one chooses matrix $\Pi$ in such a way to have $\|\Pi \beta\| \ll \|\beta\|$,
thereby $ \|\Big(\Pi \big(I_p - X^{\dag} X\big)\Big)^{\dag} \Pi \beta \| \ll \|\beta\|$, then one has
\begin{equation}
\label{beta_noiseless_c}
\nonumber
\beta \approx \Big( X^T X  + \Pi^T \Pi \Big)^{-1} X^T  X \beta
\end{equation}
Such an idea might be useful when one a-priori knows that the solution
$\beta$ belongs to some subspace of $\mathbf{R}^p$ which is given by
$\{t \in \mathbf{R}^p, G t \approx t\}$ for instance a subspace of
the form $\{t \in \mathbf{R}^p, (G -I_p) t \leq C \|t\|^2\}$ with $C=o(1) $
and $G$ standing for some linear operator which is a $p$ by $p$ matrix
different enough than the $p$ by $p$ identity matrix; in the sense that
it can compensate for the rank deficiency of matrix $X$ in the sense
that we mentioned above. As an illustrative example, assume that the
solution is smooth, then it is known that multiplying it on the left by a smoothing
kernel $G_{\sigma}$ belonging to some family of parametric kernels $\mathcal{G}= \{G_\sigma , \sigma \in \mathcal{S}\}$
(e.g. gaussian kernels with increasing bandwidth $\sigma$) would yield approximately
the same solution provided evidently that the parameter $\sigma$ of $G_{\sigma}$ is
set in such a way not to flatten too much the solution. In this case, one would take
$G_{\sigma}$ as the kernel in $\mathcal{G}$ with the smallest possible smoothing power
corresponding, say, to a parameter $\sigma^*$ and such that the augmented matrix
$\left [ \begin{array}{c}X \\ I_p - G_{\sigma^*} \end{array}\right]$ is of full rank (i.e., $p$).
One could then meet partially the identifiability condition of model (\ref{model})
with matrix $\phi$ being equal to $I_p - G_{\sigma^*}$.

We mentioned in section \ref{sec2} that one can recover, by using the proposed model selection approach,
the least-norm solution $\mu = X^{\dag} y' = X^{\dag} X\beta $ of the equation $y' = X\beta$,
in other words the solution of the optimization problem
\begin{equation}
 \|\mu\|^2 \rightarrow \min_{X\mu = X\beta} \nonumber
\end{equation}
since one has $y = X\mu + Rz $ and the linear identifiability
condition which is met for $\mu$ since one has $\mu = X^{\dag} X \mu$. One can actually generalize
such an idea in order to estimate all identifiable solutions of the form
\begin{equation}
\mu^T \Pi \mu \rightarrow \min_{\begin{array}{l} X\mu = X\beta \\ \phi \mu =0 \end{array}} \nonumber
\end{equation}
for some $p$ by $p$ positive semi-definite symmetric matrix $\Pi$ and some
$k$ by $p$ matrix $\phi$ ($k\leq p- \textrm{rank}(X)$).  Then, one checks easily that,
if matrix $\big( \Pi +  X^T X + \phi^T \phi \big)$ is of full rank, such a solution $\mu$
is unique (i.e., identifiable) and it is given by the formula
\begin{equation}
\mu = \mathcal{K} X \beta =  \mathcal{K} X \mu \nonumber
\end{equation}
where
\begin{equation}
\label{rec1}
\mathcal{K} = B \big( X^T -X^T X A X^T  - \phi^T \phi A X^T \big) + A X^T  
\end{equation}
with
\begin{equation}
A = \big( \Pi +  X^T X + \phi^T \phi \big)^{-1} \nonumber
\end{equation}
and
\begin{equation}
B= \big( X^T X + \phi^T \phi \big)^{\dag}\nonumber
\end{equation}
In practice, one may approximate formula (\ref{rec1}) for a great 
enough positive number $\mu$ as follows
\begin{equation}
\mathcal{K} \approx \mu \Big ( \Pi + \mu X^T X + \mu \phi^T \phi \Big )^{-1} X^T \nonumber
\end{equation}
Such a notion might be useful when $\beta$ belongs for instance
to some $\ell_2$-body (ellipsoid), in other words, one knows some
orthonormal matrix $\Phi$ and a positive semi-definite matrix $C$
such that
\begin{equation}
\big( \Phi \beta\big)^T C (\Phi \beta\big) \leq 1 \nonumber
\end{equation}
Hence, one might define  for example matrix $\Pi$ as follows
\begin{equation}
\Pi :=  \Phi^T C \Phi \nonumber
\end{equation}
and matrix $\phi$ might be useful if one knows that some
of the components of the vector of the coefficients $\Phi \beta$
are fatally zero, in this case, one constructs $\phi$  as follows
\begin{equation}
\phi = D\{ \delta_k \}_{k=1,p} \Phi \nonumber
\end{equation}
with $\delta_k = \mbox{1{\hskip -2.5 pt}\hbox{I}}_{(\Phi \beta)_k = 0}$. A
common example is when the coefficients of $\beta$ with respect
to some orthonormal basis $\Phi$ decay rapidly typically like a power
law, which is the case of most useful practical solutions when
the orthonormal basis $\Phi$ is appropriately chosen.\\
We would like to add that the author is currently investigating
other identifiability schemes that might broaden the scope of
application of the proposed method.

\section{A numerical study}
\label{sec5}
Please, note that our experiments below were realized with the penalty (\ref{penalty}),
and its constants were chosen as explained in subsection \ref{sub_pen}. We 
will show the performance of our method on two applications by using 
the classical performance ratio $\rho = \frac{\mathbb{E}\big[\|\tilde{\beta} - \beta\|^2\big]}{\min_{m \in \mathcal{M}}\mathbb{E}\big[\|\hat{\beta}_m - \beta\|^2\big]}$, where 
$\mathbb{E}\big[\|\tilde{\beta} - \beta\|^2\big]$ is estimated by averaging 
the value of $\|\tilde{\beta} - \beta\|^2$ over $50$ instances of the
noisy signal $y$, and $\min_{m \in \mathcal{M}}\mathbb{E}\big[\|\hat{\beta}_m - \beta\|^2\big] $
is computed directly by using formula (\ref{risk}). \\
The first application concerns 
gaussian filtering for smooth signals, and the second application concerns statistical inversion of 
ill-posed linear inverse problems by using regularity  (smoothness) and parsimony priors on the solution. 

All our experiments below were performed on the following signal (see fig. \ref{fig:fig_w}) 
and for $p=100$: 
{\small
\begin{equation}
\beta(t) = \frac{1}{50}\Big(\exp[-t](t/10)^2/2+(t/10)\log(t/10+1)+ 25\sin(t/5)\exp[t/50]\Big); \forall t= 1,\cdots,p\nonumber
\end{equation}
} 
and their Matlab code is available upon mail request to the author. 

\subsection{Gaussian smoothing}
\label{sub_exp1}
We assume an homoscedastic linear regression model
\begin{equation}
 y(t) = \beta(t) + z(t); \,\, t = 1,\cdots,p \nonumber
\end{equation}
where $z(t), t=1,\cdots,p$ are i.i.d. standard gaussian variables, and 
we consider the collection of $M$ linear filters $\big\{G_{\sigma_m}, m = 1,\cdots, M\big\}$
where for all $m = 1,\cdots,M$, $G_{\sigma_m}$ stands for a gaussian filter ($p$ by $p$
symmetric matrix) with bandwidth $\sigma_m$ defined as follows 
\begin{equation}
G_{\sigma_m}(i,j) :=  \frac{1}{a_i^m} \exp\Big[-\frac{(i-j)^2}{2\sigma_m^2}\Big]; \, \, \forall i,j =1,\cdots,p \nonumber
\end{equation}
where for all $m = 1,\cdots,M$ and for all $i=1,\cdots,p$, $a_i^m$ stands for a normalization constant 
with respect to row $i$ of $G_{\sigma_m}$ given by
\begin{equation}
a_i^m := \frac{1}{p} \sum_{j=1}^p \exp\Big[-\frac{(i-j)^2}{2\sigma_m^2}\Big] \nonumber
\end{equation}
Typically, we take  the $\sigma_m$'s as multiples of the value $\sigma = \frac{10}{M}$ as follows $\sigma_m = m \cdot\sigma, m = 1,\cdots,M$. The goal is then to select the gaussian filter with
the best bandwidth to apply on the signal $y$ by using the described
model selection approach, and the results are summarized in table \ref{table_1}
above.
\begin{table*}
\caption{Some results showing the performance ratio of the proposed method 
versus the oracle performance for the problem of signal gaussian smoothing;
for different sizes (M) of the model collection.}
\label{table_1}
\begin{tabular}{crrrrc}
\hline
Gaussian smoothing by model selection\\
M & \multicolumn{1}{c}{$\rho$} \\
\hline
$~~50$ &  1.7321 \\
$~~100$ & 1.6943 \\
$~~200$ & 1.6135\\
$~~500$ & 1.5640   \\
$~~1000$ & 1.5875   
\\ \hline
\end{tabular}
\end{table*}

\subsection{Statistical inversion of ill-posed linear inverse problems}
\label{sub_exp2}
We assume the following linear inverse problem 
\begin{equation}
 y = X \beta + z; \,\, t = 1,\cdots,p \nonumber
\end{equation}
where $z= z(t)_{t=1,\cdots,p}$ stands for a standard $p-$dimensional gaussian 
vector, and $X$ is an ill-conditioned $p$ by $p$ matrix. For our simulations,
we generated such a matrix $X$ randomly such that for all $i,j=1,\cdots,p$, 
one has $X(i,j)$ is an i.i.d. standard gaussian variable. To check the 
ill-conditioning of a randomly generated matrix $X$, we computed the ratio 
between its largest and smallest singular value $\frac{s^*}{s_*}$. For the matrix $X$ we used, we found that 
$s^* = 19.9659$ and $s_* = 0.0098$, hence  $\frac{s^*}{s_*} \approx 2043.7$.
Now, to recover $\beta$ from a a noisy observation $y$, we used 
a collection of $M$ linear filters of the form (see section \ref{sec1}) :
\begin{equation}
\label{sim_2_f}
  \Psi_m := \Big ( X^T X + \mathcal{H}_m \Big )^{-1} X^T 
\end{equation}
where, for all $m=1,\cdots,M$ , one has $\mathcal{H}_m$ which stands 
for some linear combination of first, second and third order discrete
differential operators as follows
\begin{equation}
\mathcal{H}_m = a_m \big(\mathcal{D}^1\big)^T \mathcal{D}^1 + b_m \big(\mathcal{D}^2\big)^T \mathcal{D}^2 + c_m 
\big(\mathcal{D}^3\big)^T\mathcal{D}^3\nonumber
\end{equation}
where $a_m$, $b_m $ and $c_m$ stand for three positive numbers, and 
$ \mathcal{D}^1$, $\mathcal{D}^2$  and $\mathcal{D}^3$ stand 
respectively for first, second and third order differential operators given by
\begin{equation}
\nonumber
\mathcal{D}^1(i,j) = \left \{ \begin{array}{cr} -1, &  \text{if }\, i=j. \\
                                                 1, &  \text{if}\, i=j-1. \\
                                                 0, & \text{else}.
 \end{array}\right.
\end{equation}
and
\begin{equation}
\mathcal{D}^2 = \mathcal{D}^1 \mathcal{D}^1 \, \, \,; \mathcal{D}^3 = \mathcal{D}^1 \mathcal{D}^2 \nonumber
\end{equation}
so as to achieve different degrees of smoothness in the recovered solution. \\
The sequence of the triplets $\{(a_m, b_m, c_m), m \in \mathcal{M}\}$  was
generated as follows. So, for all $m \equiv m(i,j,k)$ for some $i,j,k \in [0,\cdots,9]$
\begin{equation}
a_m := (2^i-1) \, \, ; b_m := (2^j-1) \, \, \, ; c_m := (2^k-1)  \nonumber
\end{equation}
so we ended up with a collection of $1000$ linear filters of the form (\ref{sim_2_f}). 
We repeated this experience for $50$ instances of the noisy signal $y$ in order
for us to be able to estimate the performance ratio $\rho = \frac{\mathbb{E}\big[\|\tilde{\beta} - \beta\|^2\big]}{\min_{m \in \mathcal{M}}\mathbb{E}\big[\|\hat{\beta}_m - \beta\|^2\big]}$ between the penalized
risk and the oracle risk; so we could estimate $\rho= 4.8936$. Clearly,
such a value of $\rho$ is about three times greater than in the  experiments
we showed in subsection \ref{sub_exp1} which, at first glance, may 
look like an underperformance of the presented approach. This is not actually 
the case and this needs to be moderated for the reasons that we review here. The 
first reason is that the problem we treated in this subsection is much more complicated then
the previous one since matrix $X$ is severely ill-conditioned which results in 
larger penalties, thereby in larger upper bounds of the penalized risk. The second
reason is that we used so strong priors that an oracle was enabled 
to recover the original signal almost perfectly. To see this, we computed 
for the present experiment the relative error ratio $ \frac{\min_{m \in \mathcal{M}}\mathbb{E}\big[\|\hat{\beta}_m - \beta\|^2\big]}{\|\beta\|.^2}$ and we found a value of order of $0.2\%$; 
which means that the oracle succeeded in recovering almost perfectly the actual signal $\beta$.
We compared such oracle relative error with the one of our method, and we noticed that the latter
could recover the solution with a relative error of order less than $1\%$ which is
not that bad (see figure \ref{fig:fig_2} for visual assessment of the recovered solution). Please, 
note that if one wanted to recover the solution in the present experiment by using direct inversion of the linear model as follows $X^{-1}y$, then one's expected value the relative error ratio, i.e., $\frac{\mathbb{E}\big[\|X^{-1}y - \beta\|^2\big]}{\|\beta\|.^2}$
would be of the order of $610\%$ which is of course unreasonable from a practical point of 
view (see also figure \ref{fig:fig_2} for a visual constatation).
\begin{center}
\begin{figure}
\begin{center}
{\includegraphics[width = 11cm, height = 7cm]{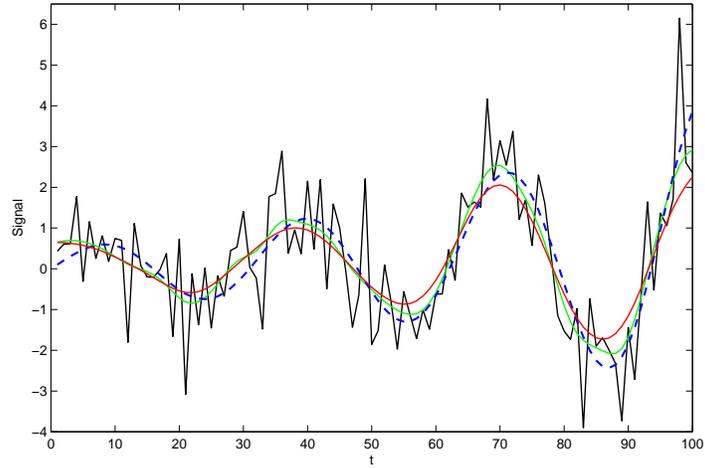}}
\caption{An instance of gaussian smoothing of a noisy signal by using the
described model selection method : the number of models is $M = 1000$;  the
oracle selected for this signal $\sigma^* = 2.43$ and the method selected $\tilde{\sigma} = 4.05$.
{\bf In dashed blue}: The original signal ($\beta$) ;  {\bf In black}:  The noisy signal (y) ;
{\bf In green} : The oracle-driven smoothed signal ($\hat{\beta}^*$) ; {\bf In red}: The data-driven smoothed signal ($\tilde{\beta}$).}
\label{fig:fig_w}
\end{center}
\end{figure}
\begin{figure}
\begin{center}
{\includegraphics[width = 11cm, height = 7cm]{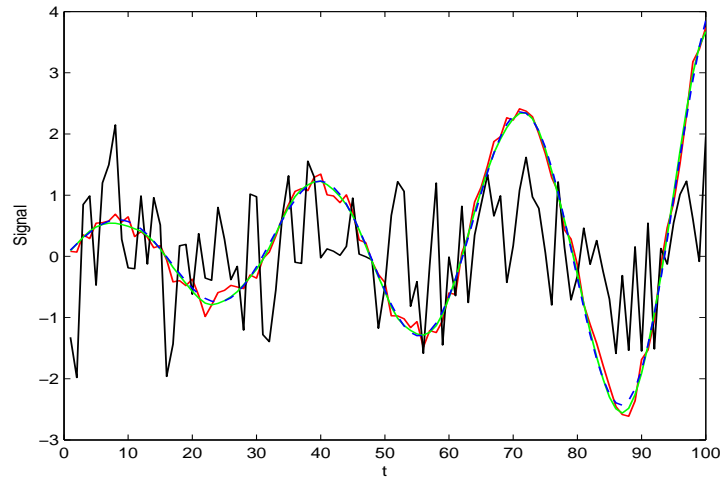}}
\caption{An instance of a statistical inversion of a linear inverse problem 
with regularity prior on the solution by using the proposed model selection approach : 
{\bf In dashed blue}: The original signal ($\beta$) ;  {\bf In black}:  The noisy signal ($\frac{y}{3}$) scaled
by a factor of $\frac{1}{3}$ to allow better visualization ; {\bf In green} : The oracle-driven estimated solution 
($\hat{\beta}^*$) ; {\bf In red}: The data-driven estimated solution ($\tilde{\beta}$).}
\label{fig:fig_2}
\end{center}
\end{figure}
\end{center}

\section{Conclusion}
We shall conclude the present paper briefly by saying that we
presented a new model selection framework that addresses the
problem of non least-squares linear estimation in linear 
regression and linear inverse problems by using model 
selection via penalization in the spirit of the pioneering 
works on non-asymptotic model selection works of Birg\'e and Massart mainly \cite{birge_massart_2007}, 
and we showed its good performance on two practical problems renowned 
to be difficult to address in practical applications. Moreover,
we think that with more optimized penalty constants (optimized 
for instance by the means of a simulation study), one could achieve even 
better performance of the proposed approach. We would like then to 
note that other papers dedicated more to applications of the proposed 
approach will follow the present one in which we plan to study more optimized 
penalties for various applications pertaining to our field 
of expertise (mainly signal and image processing applications). 

The present paper constitutes a first attempt by the 
author to give a satisfactory answer to the question of 
non-least squares estimation in regression and inverse problems
by model selection. There are, of course, still many challenges that 
we would like to address along with the statistical community
in the near future before achieving a complete final package 
of the present approach, among which we review some below in 
the form of question/answer:
\begin{itemize}
\item Could our results in this paper be improved ? Our answer to 
this question would be by "yes". Indeed, though we personally think that 
the concentration inequality that we made use of to prove 
the main theorem in this paper is sharp (see lemma \ref{lemma_1}), 
nevertheless, if one could propose sharper concentration inequalities 
that could lead to smaller penalties, one might improve on our model 
selection results in this paper. 
\item What happens if one used instead smaller penalties than the one
we proposed in this paper? In fact, one notices that theorem \ref{main_result_q} 
does not forbid the use of smaller penalties than penalty (\ref{penalty})
which we showed to enforce an oracle inequality. However, since
theorem \ref{main_result_q} only gives an upper bound of the penalized
risk, hence it is difficult to figure out actually what would be the 
behavior of the penalized estimator if one used smaller penalties mainly
for models with a value of $Q_m$ which is too big. This is an open
question that we will endeavor to solve in the near future hopefully
in the spirit of the findings of Birg\'e \& Massart in \cite{birge_massart_2007}.
\item When a linear model is underdetermined, is the linear identifiability 
assumption really mandatory in order for the model selection procedure to work? 
A rapid answer to this question would be by "yes", because such an identifiability assumption
serves in some sense as a compass for the model selection procedure
to go in the right direction in an attempt to recover the \emph{sought} 
solution of interest (among an infinite number of candidate solutions),
which turns to be rather subjective because it is imposed by the user. 
Nevertheless, we do not exclude that when some generic assumptions 
can be made on the solution of interest (by assuming for instance that 
it belongs to some restrictive class of $p-$dimensional vectors), 
one might devise for example some (implicit) identifiability schemes that 
can achieve comparable results to those in this paper.
\item When noise matrix $R$ is unknown, how one should modify 
the penalty function to allow  simultaneous estimation
of $\beta$ and $R$? We believe that this could be done
in the near future for example in the spirit of the recent 
work of Baraud and collaborators in \cite{baraud_giraud_huet_2009}.
\item What is the convergence rate of the proposed estimation approach
with respect to some classes of the vector $\beta$ (e.g. a Sobolev body)? 
We did not answer yet this question in the present paper, however, we are 
currently investing a significant amount of our time trying 
to answer this question and many other theoretical questions 
which could be of significant interest either from a theoretical or from a 
practical point of view.
\end{itemize}
By saying this, we concluded then this paper.

\section{Appendix section}\label{app}

\section*{Appendix 1 : Proof of the main theorem}
\begin{proof}
We shall use sometimes in the proof the fact that $2 a b \leq \eta a^2 + \frac{b^2}{\eta}, \forall a,b,\eta>0$.\\
Let us fix some $m \in \mathcal{M}$. One has one the one hand
\begin{multline}
\| \hat{\beta}_{m} - {\beta} \|^2  = \beta^T \big(\Psi_m X - I_p\big)^T \big(\Psi_m X - I_p\big)  \beta + z^T R^T \Psi_m^T \Psi_m R z \\+ 2   \beta^T\big(\Psi_m X - I \big)^T \Psi_m R z \nonumber
\end{multline}
and by using the fact that $\beta = \mathcal{K} X \beta$, one finds that
\begin{multline}
\| \hat{\beta}_{m} - {\beta} \|^2  = \|\beta\|^2 + (X\beta)^T \Big ( \Psi_m^T \Psi_m - 2 \mathcal{K}^T  \Psi_m \Big) X \beta + z^T R^T \Psi_m^T \Psi_m R z\\
+ 2   (X\beta)^T\big(\Psi_m  -  \mathcal{K} \big)^T \Psi_m R z\nonumber
\end{multline}
On the other hand, one has
\begin{equation}
\label{crit_quad2}
\textrm{Crit}(m) = y^T \Big ( \Psi_m^T \Psi_m  - 2 \mathcal{K}^T  \Psi_m) y + \textrm{pen}(m) \nonumber
\end{equation}
hence
\begin{multline}
\textrm{Crit}(m) = (X\beta)^T \Big ( \Psi_m^T \Psi_m  - 2 \mathcal{K}^T  \Psi_m \Big) X \beta +
z^T R^T \Big ( \Psi_m^T \Psi_m  - 2 \mathcal{K}^T  \Psi_m) R z \\
+ 2 (X\beta)^T \Big( \Psi_m^T \Psi_m  -  \mathcal{K}^T  \Psi_m - \Psi_m^T  \mathcal{K} \Big) R z + \textrm{pen}(m) \nonumber
\end{multline}
One derives that
\begin{multline}
\| \hat{\beta}_{m} - {\beta} \|^2 - \textrm{Crit}(m) = \|\beta\|^2  + 2 z^T R^T \mathcal{K}^T  \Psi_m R z
+ 2 (X\beta)^T \Psi_m^T  \mathcal{K} R z - \textrm{pen}(m) \nonumber
\end{multline}
Now, let us fix some $\theta \in (0,1)$, one finds that
\begin{multline}
(1-\theta)\| \hat{\beta}_{m} - {\beta} \|^2 - \textrm{Crit}(m) =
-\theta \beta^T \big(\Psi_m X - I_p\big)^T \big(\Psi_m X - I_p\big)  \beta + z^T R^T \Big ( 2\mathcal{K}^T  \Psi_m -\theta \Psi_m^T \Psi_m \Big ) R z \\ + 2   (X\beta)^T\Big(\Psi_m^T  \mathcal{K}  - \theta \Psi_m^T \Psi_m + \theta \mathcal{K}^T\Psi_m  \Big) R z
 + \|\beta\|^2 - \textrm{pen}(m) \nonumber
\end{multline}
hence
\begin{multline}
(1-\theta)\| \hat{\beta}_{m} - {\beta} \|^2 - \textrm{Crit}(m) =
-\theta \beta^T \Big(\Psi_m X - I_p\Big)^T \Big(\Psi_m X - I_p\Big)  \beta + z^T R^T \Big ( 2\mathcal{K}^T  \Psi_m -\theta \Psi_m^T \Psi_m \Big ) R z \\ + 2 \beta^T\Big(X^T \Psi_m^T  \mathcal{K}  - \theta X^T \Psi_m^T \Psi_m + \theta \Psi_m  \Big) R z
 + \|\beta\|^2 - \textrm{pen}(m) \nonumber
\end{multline}
and by adding and subtracting the quantity $2 \beta^T \mathcal{K} R z$ in the left side of the latter equality, one finds that
\begin{multline}
(1-\theta)\| \hat{\beta}_{m} - {\beta} \|^2 - \textrm{Crit}(m) =
-\theta \beta^T \Big(\Psi_m X - I_p\Big)^T \Big(\Psi_m X - I_p\Big)  \beta + z^T R^T \Big ( 2\mathcal{K}^T  \Psi_m -\theta \Psi_m^T \Psi_m \Big ) R z \\ - 2 \beta^T\Big(\Psi_m X - I_p\Big)^T\Big (\theta \Psi_m - \mathcal{K}\Big) R z
 + \|\beta\|^2 + 2 \beta^T \mathcal{K} R z - \textrm{pen}(m) \nonumber
\end{multline}
If now one puts $\eta_m = \big(\Psi_m X - I_p\big) \beta$ for all $m \in \mathcal{M}$, one derives
\begin{multline}
(1-\theta)\| \hat{\beta}_{m} - {\beta} \|^2 - \textrm{Crit}(m) = -\theta \|\eta_m\|^2 + z^T R^T \Big ( 2\mathcal{K}^T  \Psi_m -\theta \Psi_m^T \Psi_m \Big ) R z \\ - 2 \eta_m^T \Big (\theta \Psi_m - \mathcal{K}\Big) R z
 + \|\beta\|^2 + 2 \beta^T \mathcal{K} R z - \textrm{pen}(m) \nonumber
\end{multline}
and by definition of $\tilde{\beta} = \hat{\beta}_{\tilde{m}}$, one finds that
\begin{multline}
(1-\theta)\| \tilde{\beta} - {\beta} \|^2 = -\theta \|\eta_{\tilde{m}}\|^2 +
 z^T R^T \Big ( -\theta \Psi_{\tilde{m}}^T \Psi_{\tilde{m}} + 2\mathcal{K}^T  \Psi_m \Big ) R z \\ -
 2 \eta_{\tilde{m}}^T \Big (\theta \Psi_{\tilde{m}} - \mathcal{K}\Big) R z -  \textrm{pen}({\tilde{m}})  \\
 + \inf_{m \in \mathcal{M}} \Big\{ \|\beta\|^2 +  y^T \Big ( \Psi_m^T \Psi_m  - 2 \mathcal{K}^T  \Psi_m) y + 2 \beta^T \mathcal{K} R z + \textrm{pen}(m)    \Big\} \nonumber
\end{multline}
Since $\tilde{m}$ is random and can be any $m \in \mathcal{M}$, it follows that in order to control
$\| \tilde{\beta} - {\beta} \|^2$, one needs to control uniformly, i.e., for all $m \in \mathcal{M}$ simultaneously,
the expression
\begin{equation}
\nonumber
\Gamma_m =  z^T R^T \Big (-\theta \Psi_m^T \Psi_m + 2\mathcal{K}^T  \Psi_m  \Big ) R z \\ - 2 \eta_m^T \Big (\theta \Psi_m - \mathcal{K}\Big) R z
\end{equation}
To this end, we shall make use of concentration inequality (\ref{conc1}) of lemma (\ref{lemma_1}).
So, let us put
\begin{equation}
\mathcal{A}_m = R^T \Big (-\theta \Psi_m^T \Psi_m + \mathcal{K}^T  \Psi_m + \Psi_m^T  \mathcal{K} \Big ) R \nonumber
\end{equation}
\begin{equation}
\mathcal{B}_m = \Big (\theta \Psi_{m} - \mathcal{K}\Big) R R^T \Big (\theta \Psi_{m} - \mathcal{K}\Big)^T \nonumber
\end{equation}
and denote by $s_{m}^+$ the largest positive eigen value of the symmetric matrix $\mathcal{A}_m$, and $r_m^*$ the largest
singular value of matrix $\mathcal{B}_m$. Then, by using concentration inequality  (\ref{conc1}) of lemma (\ref{lemma_1}), 
one has for all $m \in \mathcal{M}$, with a probability larger than $1-\exp[-x_m]$, with $x_m \geq 0$ for all $m\in \mathcal{M}$, that
\begin{eqnarray}
\nonumber
Y_m &\leq& tr\big(A_m\big) + 2\sqrt{\|A_m\|^2 + 2 \eta_m^T B_m \eta_m} \sqrt{x_m} + 2 s_m^+ x_m\nonumber \\
&\leq&  tr\big(A_m\big) + 2 \|A_m\| \sqrt{x_m} + 2 \sqrt{2}\sqrt{\eta_m^T B_m \eta_m} \sqrt{x_m} + 2 s_m^+ x_m\nonumber \\
&\leq&  tr\big(A_m\big) + 2 \|A_m\| \sqrt{x_m} + \gamma_m \eta_m^T B_m \eta_m + \frac{2}{\gamma_m} x_m + 2 s_m^+ x_m\nonumber \\
&\leq&  tr\big(A_m\big) + 2 \|A_m\| \sqrt{x_m} + \gamma_m  r_m^* \|\eta_m\|^2 + 2 \big(\frac{1}{\gamma_m} + s_m^+\big) x_m\nonumber \\
\nonumber
\end{eqnarray}
for every positive number $\gamma_m$. If one takes $ r_m^* \gamma_m = \theta $, hence $\frac{1}{\gamma_m} = \frac{r_m^*}{\theta}$,
one then finds that, simultaneously for all $m\in \mathcal{M}$, with a probability larger than $1-\sum_{m\in \mathcal{M}} \exp[-x_m]$ that
\begin{multline}
(1-\theta)\| \tilde{\beta} -  {\beta} \|^2  \leq tr\big(A_{\tilde{m}}\big) + 2 \|A_{\tilde{m}} \| \sqrt{x_{\tilde{m}} } +
 2 \big(\frac{r_{\tilde{m}}^*}{\theta} + s_{\tilde{m}} ^+\big) x_{\tilde{m}}  - \textrm{pen}({\tilde{m}} ) \\
 + \inf_{m \in \mathcal{M}} \Big\{ \|\beta\|^2 +  y^T \Big ( \Psi_m^T \Psi_m  - 2 \mathcal{K}^T  \Psi_m) y + 2 \beta^T \mathcal{K} R z + \textrm{pen}(m)    \Big\}
 \nonumber
 \end{multline}
and by putting for an arbitrary positive number $\xi$: $x_m = L_m + \xi$ for all $m\in \mathcal{M}$,
one derives that, simultaneously for all $m\in \mathcal{M}$, with a probability larger than $1-\Sigma\exp[-\xi]$ that
\begin{multline}
(1-\theta)\| \tilde{\beta} -  {\beta} \|^2  \leq tr\big(A_{\tilde{m}}\big) + 2\big(\frac{r_{\tilde{m}}^*}{\theta} + s_{\tilde{m}} ^+\big) \Delta_{\tilde{m}}
+ 2 \|A_{\tilde{m}} \| \sqrt{\Delta_{\tilde{m}}+ \xi} + 2\big(\frac{r_{\tilde{m}}^*}{\theta} + s_{\tilde{m}}^+\big)\xi
 - \textrm{pen}({\tilde{m}} ) \\
 + \inf_{m \in \mathcal{M}} \Big\{ \|\beta\|^2 +  y^T \Big ( \Psi_m^T \Psi_m  - 2 \mathcal{K}^T  \Psi_m) y + 2 \beta^T \mathcal{K} R z + \textrm{pen}(m)    \Big\}
 \nonumber
 \end{multline}
We need now to separate $\xi$ from one particular $m$ by deriving for all  $m\in \mathcal{M}$ a sharp upper bound for the expression
\begin{equation}
 2 \|A_{m} \| \sqrt{L_m+ \xi} + 2\big(\frac{r_{m}^*}{\theta} + s_{m}^+\big)\xi \nonumber
\end{equation}
To do this, we consider for  $m\in \mathcal{M}$ a positive number $h_m$, then one has
$\sqrt{L_m+ \xi} - \sqrt{L_m+ h_m} = \frac{\xi - h_m}{\sqrt{L_m+ \xi} +
\sqrt{L_m+ h_m }} \leq  \frac{\xi}{\sqrt{L_m} + \sqrt{L_m + h_m}} $. Hence,
for all $\lambda >0$
\begin{eqnarray}
 2 \|A_{m} \| \sqrt{L_m+ \xi} + 2\big(\frac{r_{m}^*}{\theta} + s_{m}^+\big)\xi &\leq&  2 \|A_{m} \|\sqrt{L_m+ h_m} \nonumber \\  &+& 2 \Big ( \frac{\|A_{m} \|}{\sqrt{L_m} + \sqrt{L_m + h_m}} +  \frac{r_{m}^*}{\theta} + s_{m}^+ \Big )\xi \nonumber \\ &\leq&  2 \|A_{m} \|\sqrt{L_m+ h_m} \nonumber
 \\&+& \lambda \Big ( \frac{\|A_{m} \|}{\sqrt{L_m} + \sqrt{L_m+ h_m}} +
  \frac{r_{m}^*}{\theta} + s_{m}^+ \Big )^2 + \frac{\xi^2}{\lambda} \nonumber
\end{eqnarray}
Now, let us put for  all  $m\in \mathcal{M}$
\begin{equation}
Q_m =  tr\big(A_{m}\big) + 2\big(\frac{r_{m}^*}{\theta} + s_{m} ^+\big) L_m + 2 \|A_{m} \|\sqrt{L_m+ h_m} + \lambda \Big ( \frac{\|A_{m}\|}{\sqrt{L_m} + \sqrt{L_m + h_m}} + \frac{r_{m}^*}{\theta} + s_{m}^+ \Big )^2 \nonumber
\end{equation}
hence with a probability  larger than $1-\Sigma \exp[-\xi]$ that
\begin{multline}
(1-\theta)\| \tilde{\beta} -  {\beta} \|^2  \leq Q_{\tilde{m}} - \textrm{pen}({\tilde{m}} ) + \frac{\xi^2}{\lambda}
\\+ \inf_{m \in \mathcal{M}} \Big\{ \|\beta\|^2 +  y^T \Big ( \Psi_m^T \Psi_m  - 2 \mathcal{K}^T  \Psi_m) y + 2 \beta^T \mathcal{K} R z + \textrm{pen}(m)    \Big\}
 \nonumber
 \end{multline}
and after integration over all values of $\xi$, one finds that
\begin{multline}
(1-\theta)\mathbb{E}\Big[\| \tilde{\beta} -  {\beta} \|^2 \Big] \leq \sup_{m\in \mathcal{M}}\Big\{ Q_{m} - \textrm{pen}(m) \Big\}
 + \frac{2\Sigma}{\lambda} \\
 + \inf_{m \in \mathcal{M}} \Big\{ \|\beta\|^2 +  \beta^T  X^T \Psi_m^T \Psi_m X \beta - 2 \beta^T \Psi_m X \beta + \|\Psi_m R\|^2 - 2 tr\Big(R^T \mathcal{K}^T \Psi_m R \Big)  +  \textrm{pen}(m) \Big\} \nonumber
 \end{multline}
since
\begin{multline}
\mathbb{E}\bigg[ \inf_{m \in \mathcal{M}} \Big\{ \|\beta\|^2 +  y^T \Big ( \Psi_m^T \Psi_m  - 2 \mathcal{K}^T  \Psi_m) y + 2 \beta^T \mathcal{K} R z + \textrm{pen}(m)  \Big\}\bigg] \\
 \leq \inf_{m \in \mathcal{M}} \bigg\{ \mathbb{E}  \Big[ \|\beta\|^2 +  y^T \Big ( \Psi_m^T \Psi_m  - 2 \mathcal{K}^T  \Psi_m) y + 2 \beta^T \mathcal{K} R z + \textrm{pen}(m)  \Big]\bigg\} \\
= \inf_{m \in \mathcal{M}} \Big\{ \|\beta\|^2 +  \beta^T  X^T \Psi_m^T \Psi_m X \beta - 2 \beta^T \Psi_m X \beta + \|\Psi_m R\|^2 - 2 tr\Big(R^T \mathcal{K}^T \Psi_m R \Big)  +  \textrm{pen}(m) \Big\}
\nonumber
\end{multline}

\end{proof}

\section*{Appendix 2 : A useful concentration inequality and its proof}
{\begin{lemma}
\label{lemma_1}
Consider the random process: $T =  z^T A z + b^T z$, where $A$ is $p$ by $p$ real square matrix,
$b$ is a $p-$dimensional real vector, and $z = (z_k)_{k=1,p}$ is a $p-$dimensional standard gaussian vector,
i.e., $z_k, k=1,p$ are i.i.d. zero-mean gaussian variables with standard deviation $1$. let us denote by $s_k, k=1,p$
the respective eigen values of the symmetric matrix $\frac{1}{2}\big(A + A^T\big)$, and let us put $s^+ = \sup\{\sup_{k=1,\cdots,p}\{s_k\}, 0\}$,
and $s^{-} = \sup\{\sup_{k=1,\cdots,p}\{-s_k\}, 0\}$. Then, the following two concentration inequalities hold true for all $x>0$
\begin{equation}
\label{conc1}
\mathbb{P}\Big [T \geq tr(A)  +  2 \sqrt{\frac{1}{4}\|A + A^T\|^2 + \frac{1}{2}\|b\|^{2}} \sqrt{x} + 2 s^{+}x \Big] \leq \exp[-x]
\end{equation}
\begin{equation}
\label{conc2}
\mathbb{P}\Big [T \leq tr(A)  -  2 \sqrt{ \frac{1}{4}\|A + A^T\|^2 + \frac{1}{2}\|b\|^{2}} \sqrt{x} - 2 s^{-}x \Big] \leq \exp[-x]
\end{equation}
\end{lemma}
\begin{proof}
We already proved this lemma in \cite{bechar_2009b}, so we redo such a proof here.
We shall make use of the result of lemma (\ref{lemma_1b}) below 
to prove lemma (\ref{lemma_1}). To do this, let us consider the random process $T =  z^T A z + b^T z$, that
one can rewrite as follows: $T= \frac{1}{2} z^T \big (A+A^T\big) z + b^T z$, then by using the eigen value decomposition of
the symmetric matrix $\frac{1}{2}\big(A+A^T\big)$, one derives $T =  \sum_{k=1}^p s_k {{z'}_k^2} + {b'}_k z'$,
where $s_k, k=1,p$ stand for the respective eigen values of $\frac{1}{2}\big(A+A^T\big)$, $z' = U^T z$  with $U$
standing for the (orthonormal) eigen matrix of $\frac{1}{2}\big(A+A^T\big)$,  and $b' = U^T b$. By noticing 
that $z'$ stands for a $p-$dimensional standard gaussian vector, $\|b'\|^2 = \|b\|^2$, $\sum_{k=1}^p s_k = tr(A)$, 
and $\sum_{k=1}^p s_k^2 =  \frac{1}{4} \big\|A+A^T \big\|^2$, so by
applying lemma (\ref{lemma_1b}), the proof of lemma (\ref{lemma_1}) follows immediately.
\end{proof}

{\begin{lemma}
\label{lemma_1b}
Let $a =(a_k)_{k=1,p}$  and $b=(b_k)_{k=1,p}$ be two $p-$dimensional real
vectors, and  consider the following random expression : $T = \sum_{k=1}^p a_k z_k^2 + b_k z_k$,
where $z_k, k=1,\cdots,p$ are i.i.d. $N(0,1)$, and let us put : $a^{+} = \sup\{\sup_{k=1,\cdots,p}\{a_k\}, 0\}$,
$a^{-} = \sup\{\sup_{k=1,\cdots,p}\{-a_k\}, 0\}$. Then the following two concentration inequalities hold true for
all real positive $x$ :
\begin{equation}
\label{conc1b}
\mathbb{P}\Big [T \geq \sum_{k=1}^p a_k  +  2 \sqrt{\sum_{k=1}^p a_k^2 + \frac{b_k^2}{2}} \sqrt{x} + 2 a^{+}x \Big] \leq \exp[-x]
\end{equation}
\begin{equation}
\label{conc2b}
\mathbb{P}\Big [T \leq \sum_{k=1}^p a_k  -  2 \sqrt{\sum_{k=1}^p a_k^2 + \frac{b_k^2}{2}} \sqrt{x} - 2 a^{-}x \Big] \leq \exp[-x]
\end{equation}
\end{lemma}
\begin{proof}
This lemma was also proven in \cite{bechar_2009b}, so we redo the proof here. We shall make use 
of lemma \ref{lemma_1c} to prove lemma (\ref{lemma_1b}). Now, to prove lemma (\ref{lemma_1b}), first, one can notice that 
concentration inequality (\ref{conc2b}) can be obtained from (\ref{conc1b}) by considering the 
random quantity
\begin{equation}
T' = -T = \sum_{k=1}^p (-a_k) z_k^2 + (-b_k) z_k \nonumber
\end{equation}
and by applying (\ref{conc1b}) on $T'$ instead of $T$. So, we need to prove only (\ref{conc1b}).
To do this, let us rewrite $T$ as follows: $ T= \sum_{k=1}^p T_k$, where $ T_k =  a_k z_k ^2 + b_k z_k$,
and let us compute $\log\big [\mathbb{E}\big (\exp(y (T - \bar{T}))) \big] $, where $\bar{T} = \sum_{k=1}^p a_k$.
We have
\begin{displaymath}
\mathbb{E}\big[\exp(y T_k)\big] = \frac{1}{\sqrt{2\pi}}\int_{-\infty}^{\infty} \exp\Big [-\frac{1}{2} ((1-2 a_k y) t^2 - 2 y b_k t) \Big] dt
\end{displaymath}

\begin{displaymath}
\nonumber
\mathbb{E}\big[\exp[y T_k]\big] = \exp\bigg[ \frac{ \frac{b_k^2}{2} y^2 }{1-2 a_k y}  \bigg] \bigg (\frac{1}{\sqrt{2\pi}}\int_{-\infty}^{\infty} \exp\bigg [-\frac{1}{2} \Big ( \sqrt{1-2 a_k y} t - \frac{ b_k y}{\sqrt{1-2 a_k y}}\Big )^2 \bigg] \mathrm{d} t \bigg )
\end{displaymath}

\begin{displaymath}
\mathbb{E}\big[\exp[y T_k]\big] = \frac{\exp\bigg[ \frac{\frac{b_k^2}{2} y^2}{1-2 a_k y} \bigg] }{\sqrt{1-2 a_k y}}
\end{displaymath}

\begin{displaymath}
\mathbb{E}\Big[\exp[y (T_k - a_k )]\Big] = \frac{\exp\bigg[ \frac{\frac{b_k^2}{2} y^2 }{1-2 a_k y} \bigg] \exp[-y a_k] }{\sqrt{1-2 a_k y}}
\end{displaymath}

\begin{displaymath}
\log \Big (\mathbb{E}\Big[\exp[y (T_k - a_k )]\Big] \Big) =  \frac{\frac{b_k^2}{2} y^2}{1-2 a_k y} - \frac{1}{2}\log \Big(1-2 a_k y\Big) - a_k y
\end{displaymath}
Then, by putting $a^{+} = \sup\big\{ \sup_{k=1,\cdots,p}\{a_k\},0 \big\}$, one derives (see the technical details below) that for all
 $0<y<\frac{1}{2a^{+}}$
\begin{displaymath}
\log \Big (\mathbb{E}\Big[\exp[y (T_k - a_k )]\Big] \Big) \leq \frac{\big (a_k^2 + \frac{b_k^2}{2}\big) y^2}{1-2 a^{+} y}
\end{displaymath}
which implies by independence that for all $0<y<\frac{1}{2a^{+}}$
\begin{displaymath}
\log \Big (\mathbb{E}\Big[\exp[y (T - \bar{T} )]\Big] \Big) \leq \sum_{k=1}^{p}  \frac{\big (a_k^2 + \frac{b_k^2}{2}\big) y^2}{1-2 a^{+} y}
\end{displaymath}

\begin{displaymath}
\log \Big (\mathbb{E}\Big[\exp[y (T - \bar{T} )]\Big] \Big) \leq   \frac{\Big (\sum_{k=1}^{p} \big(a_k^2 + \frac{b_k^2}{2}\big) \Big) y^2}{1-2 a^{+} y}
\end{displaymath}
 Finally, by applying Lemma (\ref{lemma_2}) below with $ u = \sqrt{\sum_{k=1}^{p} (a_k^2 + \frac{b_k^2}{2})} $ , and $ v = 2 a^{+}$ , one derives that for all $x>0$ :
\begin{displaymath}
\mathbb{P}\Big [ T \geq  \big [\sum_{k=1}^{p} a_k \big ] +  2 \sqrt{\sum_{k=1}^{p} \big(a_k^2 + \frac{b_k^2}{2}\big)} \sqrt{x} + 2 a^{+} x \Big ]  \leq \exp[-x]
\end{displaymath}
This terminates the proof of lemma (\ref{lemma_1b})
\end{proof}

\subsection*{Some additional technical details about the proof of lemma  (\ref{lemma_1b})}
\label{tech_det}
We will show in here that for all $r>0$, $a\geq r$, and $0<y<\frac{1}{2a}$,
one has
\begin{equation}
\label{ineq_1}
 \frac{-1}{2} \log(1 - 2ry) - ry \leq \frac{r^2 y^2}{1-2{a} y}
\end{equation}
and that for all $r\leq 0$, for all $a>0$, and for all $0<y<\frac{1}{2a}$, one has
\begin{equation}
\label{ineq_2}
 -\frac{1}{2} \log(1 - 2ry) - ry \leq \frac{r^2 y^2}{1-2{a} y}
\end{equation}
\begin{proof}
let us start by showing inequality (\ref{ineq_1}). To do this, let us consider
the following function
\begin{equation}
f_{r, {a}}(y) = -\frac{1}{2} \log(1 - 2ry) - ry - \frac{r^2 y^2}{1-2{a} y} \nonumber
\end{equation}
One first notices that $f_{r, {a}}(0) =0$, then a sufficient condition for inequality
(\ref{ineq_1}) to hold true is  that $f_{r, {a}}(y)' \leq 0$, for all $0<y<\frac{1}{2a}$.
We have
\begin{displaymath}
f_{r, {a}}(y) = \frac{-1}{2} \log(1 - 2ry) - ry + \frac{r^2y}{2{a}} + \frac{r^2}{(2{a})^2} -  \frac{\frac{r^2}{(2{a})^2}}{1-2{a} y}
\end{displaymath}
one then derives that
\begin{displaymath}
f_{r, {a}}(y)' =  \frac{r}{1 - 2ry} - r +  \frac{r^2}{2{a}} -  \frac{\frac{r^2}{2 {a}}}{(1-2{a} y)^2}
\end{displaymath}
\begin{displaymath}
f_{r, {a}}(y)' =  \frac{2 r^2 y }{1 - 2ry}  -   \frac{r^2 y }{(1-2{a} y)} - \frac{r^2 y }{(1-2{a} y)^2}
\end{displaymath}
\begin{displaymath}
f_{r, {a}}(y)' \leq  \frac{2 r^2 y }{1 - 2ry}  -   \frac{2 r^2 y }{(1-2{a} y)}
\end{displaymath}
and finally since $ \frac{1 }{(1-2{r} y)}  \leq   \frac{1}{(1-2{a} y)}$, one deduces that
\begin{displaymath}
f_{r, {a}}(y)' \leq  \frac{2 r^2 y }{(1-2{a} y)}  -   \frac{2 r^2 y }{(1-2{a} y)} = 0
\end{displaymath}
then we have shown (\ref{ineq_1}).\\
We proceed in the same way as for showing inequality (\ref{ineq_1}) to show inequality (\ref{ineq_2}).
So let us consider the following function
\begin{equation}
g_{r, {a}}(y) = \frac{-1}{2} \log(1 - 2ry) - ry - \frac{r^2 y^2}{1-2{a} y} \nonumber
\end{equation}
One first notices that $g_{r, {a}}(0) =0$, then a sufficient condition for inequality (\ref{ineq_2}) to hold
true is to that $g_{r, {a}}(y)' \leq 0$ for all $0<y<\frac{1}{2a}$. One derives that
\begin{displaymath}
g_{r, {a}}(y)' =  \frac{r}{1 - 2ry} - r +  \frac{r^2}{2{a}} -  \frac{\frac{r^2}{2 {a}}}{(1-2{a} y)^2}
\end{displaymath}

\begin{displaymath}
g_{r, {a}}(y)' =  \frac{2 r^2 y }{1 - 2ry}  -   \frac{r^2 y }{(1-2{a} y)} - \frac{r^2 y }{(1-2{a} y)^2}
\end{displaymath}
\begin{displaymath}
g_{r, {a}}(y)' \leq  \frac{2 r^2 y }{1 - 2ry}  -   \frac{2 r^2 y }{(1-2{a} y)}
\end{displaymath}
and finally, since $\frac{1}{1 - 2ry} \leq  \frac{1}{(1-2{a} y)} $, one finds that
\begin{displaymath}
g_{r, {a}}(y)' \leq  \frac{2 r^2 y }{(1-2{a} y)}  -   \frac{2 r^2 y }{(1-2{a} y)} = 0
\end{displaymath}
\end{proof}

\subsection*{Birge's \& Massart concentration inequality}
\begin{lemma}
\label{lemma_1c}
If a random variable $\xi$ satisfies for some two real positive numbers $u$ and $v$ the following inequality :
\begin{equation}
\log\Big( \mathbb{E}\Big[\exp[ y \xi ] \Big] \Big) \leq \frac{(u y)^2}{1-v y}, \textrm{for all \,\,\,} 0<y<\frac{1}{v}
\end{equation}
then
\begin{equation}
\mathbb{P}\Big[ \xi \geq 2 u\sqrt{x} + v x \Big] \leq \exp[-x], \textrm{for all \,\,\,} x>0
\end{equation}
\end{lemma}
The proof of this lemma can be found in  \cite{birge_massart_1998}.

\section*{Acknowledgements}
This work has been accomplished by the author jointly at the Oxford Centre for Magnetic Resonance
Imaging of Oxford University (U.K.), and at INRIA Nice Sophia Antipolis (France). The author is
very grateful to : Alain Trubuil (INRA of Jouy-en-Josas), Christine Graffigne (MAP5, Universite Paris 5),
Sylvie Huet (INRA of Jouy-en-Josas), Matthew Robson (OCMR \& FMRIB, Oxford University),
Alison Noble (BioMedIA, Oxford University) and Ilias Kylintireas (OCMR, Oxford University)
for all the helpful discussions either about statistical model selection or image processing in
video-microscopy and MRI during my two scientific stays at INRA of Jouy-en-Josas and at OCMR/FMRIB
of Oxford University respectively.

\end{document}